\documentclass[11pt]{article}

\usepackage[utf8]{inputenc}
\usepackage[english]{babel}

\usepackage[english]{babel}
\usepackage{amsfonts}
\usepackage{amstext}
\usepackage[usenames]{color}
\usepackage{graphicx}
\usepackage{amsmath}
\usepackage{amssymb}
\usepackage{amsthm}
\usepackage{appendix}
\usepackage{euscript}
\usepackage{enumerate}

\usepackage{color}

\usepackage{chngcntr}
\usepackage{apptools}

\numberwithin{equation}{section}
  
\graphicspath{{Graphics/}}
\newtheorem{definition}{Definition}
\newtheorem{lemma}{Lemma}
\newtheorem{theorem}{Theorem}
\newtheorem{proposition}{Proposition}
\newtheorem{remark}{Remark}
\newtheorem{corollary}{Corollary}

\AtAppendix{\counterwithin{definition}{section}}
\AtAppendix{\counterwithin{lemma}{section}}
\AtAppendix{\counterwithin{theorem}{section}}
\AtAppendix{\counterwithin{proposition}{section}}
\AtAppendix{\counterwithin{remark}{section}}
\AtAppendix{\counterwithin{corollary}{section}}

\renewcommand{\Im}{\operatorname{Im}}

\newcommand{\ord}{\mathrm{O}}
\newcommand{\osmall}{\mathrm{o}}
\renewcommand{\Re}{\operatorname{Re}}
\newcommand{\I}{\mathrm{i}}

\begin{document}
\title{Global conservative solutions of the nonlocal NLS equation beyond blow-up}
\author{Yan Rybalko$^{\dag}$ and Dmitry Shepelsky$^{\dag}$\\
 \small\em {}$^\dag$ B.Verkin Institute for Low Temperature Physics and Engineering of the National Academy of Sciencies of Ukraine}

\date{}

\maketitle

\begin{abstract}
We consider the Cauchy problem for 
the integrable nonlocal nonlinear Schr\"odinger (NNLS) equation
$
\I\partial_t q(x,t)+\partial_{x}^2q(x,t)+2\sigma q^{2}(x,t)\overline{q(-x,t)}=0
$
with initial data $q(x,0)\in H^{1,1}(\mathbb{R})$.
It is known that the NNLS equation is integrable and it has soliton solutions, which can have isolated finite time blow-up points.
The main aim of this work is to propose a suitable concept for continuation of  weak $H^{1,1}$ local solutions of the general 
Cauchy problem (particularly, those
admitting  long-time soliton resolution) beyond possible singularities.
Our main tool is the inverse scattering transform method in the form of the Riemann-Hilbert problem combined with the PDE existence theory for nonlinear dispersive equations.
\end{abstract}

\section{Introduction}
\label{intr}
We consider the initial value problem for a nonlocal nonlinear Schr\"odinger (NNLS) equation
\begin{subequations}\label{iv-nnls}
	\begin{align}\label{nnls}
	&\I\partial_t q(x,t)+\partial_{x}^2q(x,t)+2\sigma q^{2}(x,t)\overline{q(-x,t)}=0,
	\quad x,t\in\mathbb{R},\quad \sigma=\pm 1, \\
	\label{iv}
	&q(x,0)=q_{0}(x),
	\end{align}
\end{subequations}
with the initial data $q_0(x)\in H^{1,1}(\mathbb R)$. Here and below,
$\bar{q}$ denotes the complex conjugate of $q$ and $H^{s,m}(\mathbb{R})$ denotes the weighted Sobolev space:
$$
H^{s,m}(\mathbb{R})=
\left\{
f(x): f,\partial^s_x f, x^m f\in L^2(\mathbb{R})
\right\},\quad s,m\geq 0,
$$
with the norm
\begin{equation}
\|f\|_{H^{s,m}}^2=\|f\|^2_{L^2}+\|\partial^s_x f\|^2_{L^2}+
\|x^mf\|^2_{L^2}.
\end{equation}
By analogy with the conventional  (local) nonlinear Schr\"odinger (NLS) equations, we will call \eqref{nnls} with $\sigma=1$ and $\sigma=-1$ the focusing and defocusing NNLS equation respectively.
Notice, however, that the nonlocal equations \eqref{nnls} have substantially different properties compared with their local counterparts.
For example, the linear operator associated with the defocusing NNLS can have discrete spectrum \cite{AMN} (see also Proposition \ref{defoc-zeros} below), while the focusing NNLS supports both dark and bright soliton solutions simultaneously \cite{SMMC14}.

The NNLS equation was introduced by Ablowitz and Musslimani in \cite{AMP} as a new nonlocal reduction of the Ablowitz-Kaup-Newell-Segur (AKNS) system \cite{AS} that preserves the \textit{PT}-symmetry property \cite{BB}: if $q(x,t)$ is a solution of \eqref{nnls}, so is $\overline{q(-x,-t)}$.
Moreover, this equation can be viewed as an example of the so-called Alice-Bob systems \cite{L18}, where the evolution of the solution depends on the state at non-neighboring points.
Also notice that the NNLS equation is gauge equivalent to the coupled Landau-Lifshitz equation \cite{GA, R21}, so it can be useful in the theory of magnetism.

The initial value problem \eqref{iv-nnls} was firstly considered in \cite{AMN}, where the authors developed the inverse scattering transform (IST) method and obtained exact formulas for one and two soliton solutions (see also \cite{MHW20} for the generalization to  nonlocal NLS hierarchies).
The general formulas for multisoliton solutions on the  zero background were obtained by a number of authors using different approaches: in \cite{FLAM18}, a mixture of  Hirota's bilinear method and the Kadomtsev--Petviashvili hierarchy reduction method; in \cite{Y19}, the Riemann-Hilbert (RH) approach;
in \cite{MS19}, the B\"acklund-Darboux transformation;
in \cite{WZ22}, the Dbar method.
Moreover, in \cite{Z19} the relation between soliton solutions and a motion of curves in $\mathbb{C}^3$ was discussed, and in \cite{GS} the authors proved the complete integrability of \eqref{iv-nnls}.

The focusing NNLS equation has a family of one soliton solutions with smooth and rapidly decaying initial profile
(see Figure \ref{one-soliton-f}):
\begin{equation}\label{one-soliton}
\begin{split}
q_1^{\mathrm{sol}}(x,t)=&
\frac{2(\rho_{1,1}-\rho_{2,1})}
{\gamma_{2,1}^{-1}e^{-2\rho_{2,1}x-4\I\rho_{2,1}^2t}
	-\gamma_{1,1}e^{-2\rho_{1,1}x-4\I\rho_{1,1}^2t}},\\
&\text{with}
\quad (-1)^{j+1}\rho_{j,1}>0,\,|\gamma_{j,1}|=1,\,j=1,2.
\end{split}
\end{equation}
If $\rho_{1,1}\neq-\rho_{2,1}$, then the
solutions blow-up
at $(x,t)=(0,t_n)$, $n\in\mathbb{Z}$ with
\begin{equation}
t_n=\frac{\arg(\gamma_{1,1}\gamma_{2,1})+2\pi n}
{4(\rho_{1,1}^2-\rho_{2,1}^2)}.
\end{equation}
The existence of  singular solutions like  \eqref{one-soliton} implies  that the well-posedness theory for the Cauchy problem \eqref{iv-nnls} requires a careful investigation.
This question was  addressed by Genoud in \cite{G17}, where the author proved the existence and uniqueness of a \textit{local} weak $H^1$ solution.
Regarding the global existence, Genoud showed that there exists an arbitrary small (in $H^1$) soliton profile of the form \eqref{one-soliton}, whose evolution exhibits a blow-up in a finite time.
The main technical difficulty for proving global well-posedness in $H^{1}$
is due to the fact that
the conservation laws, being nonlocal,  do not preserve any reasonable norm and can be negative (see \eqref{cons-mass} and \eqref{cons-energy} below).

\begin{figure}[h]
	\begin{minipage}[h]{0.99\linewidth}
		\centering{\includegraphics[width=0.79\linewidth]
			{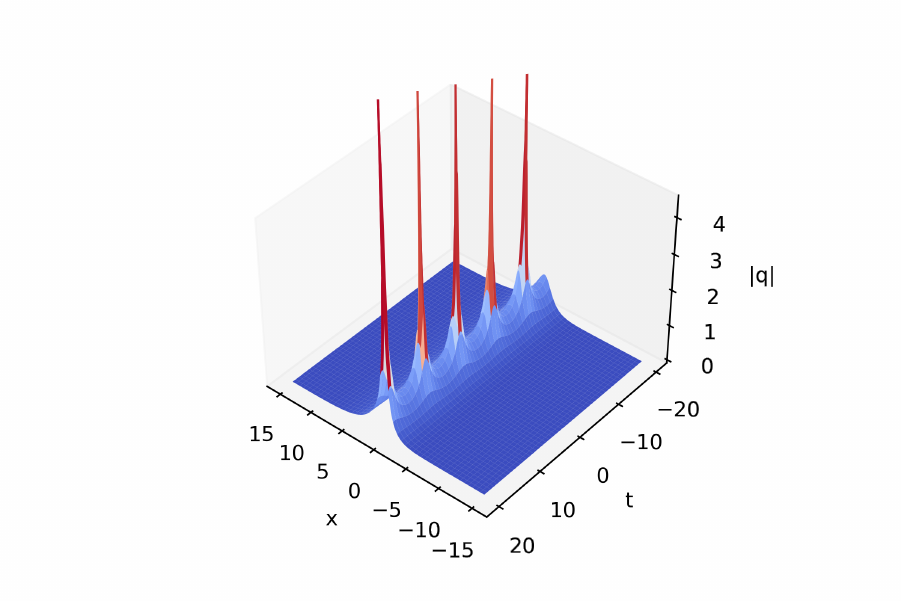}}
		\caption{Modulus of soliton \eqref{one-soliton} with
			$\rho_{1,1}=\frac{1}{2}$,
			$\rho_{2,1}=-\frac{1}{4}$,
			$\gamma_{1,1}=e^{\I\frac{\pi}{2}}$ and
			$\gamma_{2,1}=e^{\I\frac{\pi}{6}}$.}
		\label{one-soliton-f}
	\end{minipage}
\end{figure}

Some steps towards the development of the global theory for the NNLS equation have recently been made, see \cite{CLW22} and \cite{ZF22}.
In \cite{CLW22} the authors proved the global existence of the solution in the super-critical function spaces $E_\sigma^s$, where $s<0$ and $\sigma>-\frac{1}{2}$, with an additional requirement that the support of the Fourier transform of the initial data lies in $[\varepsilon_0,\infty)$ for some $\varepsilon_0>0$.
Notice that though such class of initial data involves rather rough functions, it does not contain one soliton (smooth) initial profile of \eqref{one-soliton}.
In \cite{ZF22} the authors used the IST approach for obtaining global weak solutions in the weighted Sobolev space $H^{1,1}$ assuming that the initial data is small (cf.\,\,Theorem \ref{H-gl-sol} below).
Unlike the case of, for example, the defocusing NLS equation \cite{DZ03} or the derivative NLS equation \cite{JLPS20, LYF22, PSS17, PS18}, the jump matrix of the associated RH problem for the NNLS equation does not satisfy appropriate symmetries
which would imply
the solvability of this RH problem \cite{Z89-1}.
Therefore, to ensure that a solution to the RH problem exists, \cite{ZF22} imposes a small norm assumption
on the initial data, which rules out any soliton solutions, particularly, \eqref{one-soliton}, see Remark \ref{H1-H-11} below.

Though the one soliton solution \eqref{one-soliton} does not exist globally as a bounded function, it seems natural to consider this solution in the whole plane $(x,t)$ outside  the blow-up points $(0,t_n)$, $n\in\mathbb{Z}$.
Indeed, the function \eqref{one-soliton} decays exponentially fast as $x\to\pm\infty$ for any $t\in\mathbb{R}$ and it satisfies the focusing NNLS equation for all $(x,t)\neq(0,t_n)$.
So, the one soliton (as other types of exact solutions with singularities, see \cite{MS19}) can be continued after  the blow-up points directly by its exact formula.
Then a natural question arises how to understand solutions, with blow-up points, of the general Cauchy problem, especially with initial data in Sobolev spaces.


The problem of continuation of a solution of an evolution  PDE after blow-up points has a long history and it involves a number of different approaches, see, e.g., \cite{GMP15}.
Here we mention the related problem for the $L^2$-critical NLS equation, where the blow-up solution $q(\cdot, t)$ is approximated by  globally defined solutions $q_\delta(\cdot, t)$ as $\delta\to0$ \cite{M92-CMP}.
Also such questions arise for  integrable Camassa-Holm-type equations.
For these systems, one  method for obtaining global solutions is the reformulation of  the original problem in the form of a semilinear system of ODEs with values in a Banach space \cite{BC07}.
We refer the reader to a review article \cite{LS22} and references therein for more results on wave breaking phenomena and global solutions for peakon equations.

The goal of this paper is to continue a solution of \eqref{iv-nnls} beyond the blow-ups  for a class of initial data in $H^{1,1}(\mathbb{R})$, which includes  multisoliton solutions.
To this end we use the IST method in the form of a Riemann-Hilbert problem,
which results in obtaining the reconstruction formula for the solution $q(x,t)$, formulated in terms of the associated spectral data,
see \eqref{q(x,t)} below.
The solution of the RH problem can be evaluated separately  for (almost) all $x,t\in\mathbb{R}$; such a ``locality'' allows us to circumvent  possible blow-up points as
in the case of the soliton solutions \eqref{one-soliton}.
Namely, we determine
a set of blow-up points $(x,t)\in\mathcal{B}\subset\mathbb{R}^2$ as solutions of a system of two (nonlinear) algebraic equations written in terms of the solution of the associated RH problem (see \eqref{B-L-sol} and Remark \ref{char-B} below).
For all other points,
$(x,t)\in\mathbb{R}^2\setminus\mathcal{B}$,
we can evaluate (under certain conditions on $q_0$) the solution of the RH problem and, consequently, $q(x,t)$, which has the required smoothness and decay properties.

Moreover, since $q(x,t)$ is defined via the inverse transform, we have, almost ``for free'', that the solution continued after a blow-up  is conservative, i.e., for all $t\in\mathbb{R}$ such that $(\mathbb{R}\times\{t\})\cap\mathcal{B}=\emptyset$, the total ``mass'' $\tilde{M}$ and ``energy'' $\tilde{E}$
associated with $q(\cdot,t)$
are the same as those related to  $q_0(\cdot)$  (see Remark \ref{concl-th-1}, Item (i) and Remark \ref{concl-th-2}, Item (i) below).
Another advantage of using the RH approach is the efficiency of numerical calculations.
The computational cost of computing the solution of the RH problem does not depend on the values of $x$ and $t$ (the ``locality'' of the RH problem is manifested here as well), so it is possible to calculate numerically the set of blow-up points $\mathcal{B}$ and then evaluate the solution $q(x,t)$ both at regular finite  points  $(x,t)$ and asymptotically \cite{TO16}.

The article is organized as follows.
Following \cite{AMN} and \cite{RS19}, in Section \ref{ist-RH} we develop  the IST method for the Cauchy problem \eqref{iv-nnls}.
In Section \ref{sm-iv} we consider problem \eqref{iv-nnls} with $q_0(x)$ in the Schwartz space.
In Theorem \ref{th-class-glob} we prove existence and uniqueness of the classical solution
in the case of small initial data. Then in Theorems \ref{th-class-one-sol} and \ref{th-class-L-sol} we construct the global solution with possible finite time blow-up points for small perturbations of one and multisoliton initial profiles respectively, which we call  \textit{blowing-up Schwartz solutions} (see Definition \ref{blow-up-S}).
These solutions are smooth, rapidly decay as $x\to\pm\infty$,  and satisfy the NNLS equation at every regular point $(x,t)$.
In Section \ref{bl-up-H11}, we first prove the global existence and uniqueness of weak $H^{1,1}$ solutions for small initial data, see Theorem \ref{H-gl-sol} (cf.\,\,\cite{ZF22}),
and then
formulate the concept of a \textit{pointwise $H^{1,1}$ solution} (see Definition \ref{blow-up-H11}), which can have finite time blow-up points.
Then we provide a global pointwise $H^{1,1}$ solution for $q_0(x)\in H^{1,1}(\mathbb{R})$ in the neighborhood of one and multisoliton initial profiles and prove that this solution can be uniformly, in a sense, approximated  by  blowing-up Schwartz solutions (see Theorems \ref{th-H11-one-sol} and \ref{th-H11-L-sol}), which is the main result of this paper.


\section{IST method in the form of the Riemann-Hilbert problem}\label{ist-RH}

The IST method was firstly developed for the NNLS equation with zero background in \cite{AMN} (see also \cite{GS, RS19}).
Here we briefly recall the scheme of the direct and inverse transform for  problem \eqref{iv-nnls}.

\subsection{Direct transform}

Consider the system of two linear equations (Lax pair) for a $2\times2$-valued function $\Phi(x,t,k)$:
\begin{subequations}\label{Lax}
	\begin{align}\label{ZS}
	&\partial_x\Phi+\I k\sigma_{3}\Phi=U\Phi,\\
	\label{t_evol}
	&\partial_t\Phi+2\I k^{2}\sigma_{3}\Phi=V\Phi,
	\end{align}
\end{subequations}
where $\sigma_3=
\left(
\begin{smallmatrix}
1& 0\\
0 &-1
\end{smallmatrix}
\right)
$
is the third Pauli matrix,
\begin{equation}
U=\begin{pmatrix}
0& q(x,t)\\
-\sigma \overline{q(-x,t)}& 0\\
\end{pmatrix},
\quad
V=\begin{pmatrix}
V_{11}& V_{12}\\
V_{21}& -V_{11}\\
\end{pmatrix},
\end{equation}
with $V_{11}=\I \sigma q(x,t)\overline{q(-x,t)}$, $V_{12}=2kq(x,t)+\I\partial_x q(x,t)$, and
$V_{21}=-2k\sigma \overline{q(-x,t)}+\I \sigma \partial_x(\overline{q(-x,t)})$.
Then the integrable NNLS equation is equivalent to the compatibility condition of the system \eqref{Lax} \cite{AMP}, i.e.
$$
\text{equation }
\eqref{nnls}\Leftrightarrow
\partial_t\hat U-\partial_x\hat V+[\hat U, \hat V]=0,
$$
where $\hat U=U-\I k\sigma_3$, $\hat V=V-2\I k^2\sigma_3$ and $[A,B]=AB-BA$ is the matrix commutator.
In this sense, system \eqref{Lax} linearizes equation \eqref{nnls}.

Define two (Jost) solutions $\Phi_1$ and $\Phi_2$ of \eqref{Lax}, which are normalized as follows:
$$
\Phi_j(x,t,k)e^{(\I kx+2\I k^{2}t)\sigma_{3}}\to I,\quad x\to (-1)^j\infty,\quad t,k\in\mathbb{R},\quad j=1,2,
$$
where $I$ is the $2\times2$ identity matrix.
It can be shown that if $q(\cdot, t)\in L^1(\mathbb{R})$ for all $t\in\mathbb{R}$, then $\Phi_j(x,t,k)=\Psi_j(x,t,k)e^{(-\I kx-2\I k^{2}t)\sigma_{3}}$,  where $\Psi_j$, $j=1,2$ are the unique solutions of the following Volterra integral equations:
\begin{subequations}\label{Psi}
	\begin{align}\label{Psi-1}
	&\Psi_1(x,t,k)=I+\int_{-\infty}^{x}e^{\I k(y-x)\sigma_3}U(y,t)\Psi_1(y,t,k)e^{-\I k(y-x)\sigma_3}\,dy,\quad
	k\in\mathbb{R}, \\
	\label{Psi-2}
	&\Psi_2(x,t,k)=I+\int_{+\infty}^{x}e^{\I k(y-x)\sigma_3}U(y,t)\Psi_2(y,t,k)e^{-\I k(y-x)\sigma_3}\,dy,\quad
	k\in\mathbb{R}.
	\end{align}
\end{subequations}
Since $U(x,t)$ has zero trace, we have $\det\Psi_j=\det\Phi_j=1$ for all $x,t,k\in\mathbb{R}$.
Then $\Phi_1$ and $\Phi_2$,
being two different solutions of the differential equations \eqref{Lax}, can be related by a matrix $S(k)$ with $\det S(k)=1$:
\begin{equation}
\label{S-def}
\Phi_1(x,t,k)=\Phi_2(x,t,k)S(k),\quad k\in\mathbb{R}.
\end{equation}

Let $\Lambda:=\left(
\begin{smallmatrix}
0 & \sigma\\1 & 0
\end{smallmatrix}\right)$,
$\mathbb{C}^{\pm}=\left\{k\in\mathbb{C}\,|\pm\Im k>0\right\}$, and denote by $A^{[j]}$  the $j$-th column of the matrix $A$.
From the integral equations \eqref{Psi} it follows that $\Psi_1^{[1]}(x,t,k)$ and $\Psi_2^{[2]}(x,t,k)$ can be analytically continued for $k\in\mathbb{C}^+$ whereas $\Psi_1^{[2]}(x,t,k)$ and $\Psi_2^{[1]}(x,t,k)$ can be analytically continued for $k\in\mathbb{C}^-$.
Moreover, the symmetry $\Lambda \overline{U(-x,t)}\Lambda^{-1} = U(x,t)$ implies the connections between $\Phi_1$ and $\Phi_2$:
\begin{subequations}\label{phi-sym}
	\begin{align}
	&\Lambda\overline{\Phi_1(-x,t,-k)}\Lambda^{-1}=\Phi_2(x,t,k),&& x,t,k\in\mathbb{R},\\
	\label{phi-sym-b}
	&\Lambda\overline{\Phi_1^{[1]}(-x,t,-\bar k)}=\Phi_2^{[2]}(x,t,k),&& x,t\in\mathbb{R},\,k\in\mathbb{C}^{+},\\
	\label{phi-sym-c}
	&\Lambda\overline{\Phi_1^{[2]}(-x,t,-\bar k)}=\Phi_2^{[1]}(x,t,k),&& x,t\in\mathbb{R},\,k\in\mathbb{C}^{-}.
	\end{align}
\end{subequations}
In turn, from \eqref{S-def} and \eqref{phi-sym} we conclude that the scattering matrix $S(k)$ can be written in the form
\begin{equation}
\label{S(k)}
S(k)=\begin{pmatrix}
a_{1}(k)& -\sigma\overline{b(-k)}\\
b(k)& a_2(k)\\
\end{pmatrix}
,\quad k\in\mathbb{R},
\end{equation}
with some $b(k)$, $a_1(k)$, and $a_2(k)$, which are referred to as the spectral (or scattering) functions.
We summarize some properties of $a_1, a_2$ and $b$ in Proposition \ref{a_j,b} (we use the notation $\overline{\mathcal{D}}$ for the closure of the set $\mathcal{D}$).
\begin{proposition}\label{a_j,b}\cite{AMN}
	The spectral functions $b(k)$, $a_1(k)$ and $a_2(k)$ have the following properties:
	\begin{enumerate}
		\item Analyticity:
		$a_{1}(k)$ is analytic in $k\in\mathbb{C}^{+}$
		and continuous in
		$\overline{\mathbb{C}^{+}}$;
		$a_{2}(k)$ is analytic in $k\in\mathbb{C}^{-}$
		and continuous in
		$\overline{\mathbb{C}^{-}}$.
		\item Asymptotics for the large $k$:
		$a_{1}(k)=1+\ord\left(k^{-1}\right)$, $k\to\infty$, $k\in\overline{\mathbb{C}^+}$;
		$a_{2}(k)=1+\ord\left(k^{-1}\right)$, $k\to\infty$,
		$k\in\overline{\mathbb{C}^-}$;
		$b(k)=\ord\left(k^{-1}\right)$, $k\to\infty$, $k\in\mathbb{R}$.
		\item Symmetries:
		$\overline{a_{1}(-\bar{k})}=a_1(k)$,
		$k\in\overline{\mathbb{C}^{+}}$;
		$\overline{a_{2}(-\bar{k})}=a_2(k)$,
		$k\in\overline{\mathbb{C}^{-}}$.
		\item Determinant relation:
		$a_{1}(k)a_{2}(k)+\sigma b(k)\overline{b(-k)}=1$, $k\in{\mathbb R}$.
	\end{enumerate}
\end{proposition}
\begin{proof}
	Item 1.
	Using the Cramer's rule for the system \eqref{S-def} we conclude that $a_1$, $a_2$ and $b$ can be found in terms of the following determinant relations, where we can take any $x,t\in\mathbb{R}$:
	\begin{subequations}\label{det-rel}
		\begin{align}
		\label{a1-det}
		&a_1(k)=\det(\Phi_1^{[1]}(x,t,k), \Phi_2^{[2]}(x,t,k)),&&
		k\in\overline{\mathbb{C}^+},\\
		&a_2(k)=\det(\Phi_2^{[1]}(x,t,k), \Phi_1^{[2]}(x,t,k)),&&
		k\in\overline{\mathbb{C}^-},\\
		&b(k)=\det(\Phi_2^{[1]}(x,t,k),\Phi_1^{[1]}(x,t,k)),&&
		k\in\mathbb{R}.
		\end{align}
	\end{subequations}
	Item 2. Follows from \eqref{det-rel} and $\Phi_j(0,0,k)\to I$, $k\to\infty$ for $q_0(x)\in L^1(\mathbb{R})$, which, in turn, follows from \eqref{Psi}.\\
	Item 3. Follows from \eqref{S-def} and \eqref{phi-sym}.\\
	Item 4. Follows from $\det S(k)=1$ and \eqref{S(k)}.
\end{proof}
\begin{remark}
	From \eqref{det-rel} it follows that the spectral functions can be found in terms of the known initial data $q_0(x)=q(x,0)$ only.
\end{remark}
\begin{remark}
	$a_1$, $a_2$ and $b$ can  also be found as the large $x$ limits of the solutions of the Volterra equations (see, e.g., Section 1.3 in \cite{AS} or \cite{RS19} for details):
	\begin{equation}
	\label{ab-lim}
	\begin{split}
	a_{1}(k)&=\lim\limits_{x\rightarrow+\infty}\psi_1(x,t,k),\quad
	a_{2}(k)=\lim\limits_{x\rightarrow+\infty}\psi_4(x,t,k),\\
	b(k)&=\lim\limits_{x\rightarrow+\infty}
	e^{-2\I kx-4\I k^2t}\psi_3(x,t,k),
	\end{split}
	\end{equation}
	where $\{\psi_1, \psi_2, \psi_3, \psi_4\}$ solves the following integral equations (notice that\break $\Psi_1(x,t,k)\equiv
	\left(
	\begin{smallmatrix}
	\psi_1& \psi_2\\
	\psi_3& \psi_4
	\end{smallmatrix}
	\right)
	$):
	\begin{equation}\label{psi-1-3}
	\begin{cases}
	\psi_{1}(x,t,k)=1+\int_{-\infty}^{x}q(y,t)\psi_3(y,t,k)\,dy,\\
	\psi_3(x,t,k)=-\sigma\int_{-\infty}^{x}e^{2\I k(x-y)}\overline{q(-y,t)}\psi_1(y,t,k)\,dy,
	\end{cases}
	\end{equation}
	and
	\begin{equation}
	\label{psi24}
	\begin{cases}
	\psi_2(x,t,k)=\int_{-\infty}^{x}e^{2\I k(y-x)}q(y,t)\psi_4(y,t,k)\,dy,\\
	\psi_4(x,t,k)=1-\sigma\int_{-\infty}^{x}
	\overline{q(-y,t)}\psi_2(y,t,k)\,dy.
	\end{cases}
	\end{equation}
\end{remark}
An important property of the spectral functions $a_1(k)$ and $a_2(k)$ is that they can have  zeros in the corresponding closed complex half plane.
These zeros correspond to solitons, which can have isolated, finite time blow-up points.
Recall that in the case of the conventional defocusing NLS equation, there are no discrete spectrum (i.e.,\,\,no zeros) and, consequently, there are no solitons.
However, for the NNLS equation, the situation is different.
Indeed, consider an odd initial data for the defocusing NNLS equation. Then problem \eqref{iv-nnls} reduces to  the Cauchy problem for the conventional focusing  NLS equation, which can have discrete spectrum and therefore supports solitons.
But for the defocusing NNLS equation there are still restrictions on the zeros of $a_1$ and $a_2$ (cf.\,\,Section 9 in \cite{AMN}):
\begin{proposition}\label{defoc-zeros}
	If $\sigma=-1$, the spectral functions $a_1(k)$	and $a_2(k)$ do not have purely imaginary zeros.
	\begin{proof}
		We give the proof for $a_1(k)$, the case of $a_2(k)$ is similar.
		First, in view of the determinant relation (see Proposition \ref{a_j,b}, Item 4), $a_1(0)\neq0$.
		Suppose that $a_1(\I\rho)=0$, $\rho>0$.
		Then \eqref{a1-det} implies that there exists $\gamma\in\mathbb{C}$ such that $\Phi_1^{[1]}(x,t,\I\rho)=\gamma\Phi_2^{[2]}(x,t,\I\rho)$.
		Notice that since $\Phi_1^{[1]}$ and $\Phi_2^{[2]}$ are nonzero solutions of \eqref{Lax}, $\gamma_x=\gamma_t=0$ for all $x,t\in\mathbb{R}$.
		Moreover, from the symmetry relations \eqref{phi-sym-b}, \eqref{phi-sym-c} we conclude that
		$$
		\left(\Phi_{1}\right)_{11}(x,t,\I\rho)=
		\gamma\sigma\overline{\left(\Phi_{1}\right)_{21}(-x,t,\I\rho)},\quad
		\left(\Phi_{1}\right)_{21}(x,t,\I\rho)=
		\gamma\overline{\left(\Phi_{1}\right)_{11}(-x,t,\I\rho)}.
		$$
		The latter implies that $\sigma|\gamma|^2=1$, which is inconsistent with $\sigma=-1$.
	\end{proof}
\end{proposition}
In this paper we make the following assumption on zeros of $a_j(k)$, $j=1,2$:

\noindent\textbf{Assumption 1.} (Zeros of $a_1$ and $a_2$). \emph{We assume that neither $a_1(k)$ nor $a_2(k)$ has zeros on the real axis (i.e., there are no spectral singularities).
	Moreover, we suppose that both $a_1(k)$ and $a_2(k)$ can have only finite number of simple zeros in $\mathbb{C}^+$ and $\mathbb{C}^-$ respectively.}
\begin{remark}
	In the case of the conventional focusing NLS equation, the spectral functions corresponding to ``generic'' initial data (i.e., data that belong to an open dense set in a certain functional space) satisfy Assumption 1.
	On the other hand, there exists Schwartz initial data  such that  the corresponding spectral functions have infinitely many  non-real zeros with accumulation  points on the real line \cite{Z89}.
	Concerning the nonlocal case, it is an open problem whether the initial data which satisfy Assumption 1 forms a dense set in a reasonable topology.
\end{remark}


Another distinctive feature of the NNLS equation compared to  its conventional counterpart is that $a_1(k)$ and $a_2(k)$ are not related and therefore the number of zeros in $\mathbb{C}^+$ and $\mathbb{C}^-$ can be different.
Indeed, Item 3 of Proposition \ref{a_j,b}  implies that  the zeros are symmetric w.r.t.\,\,the imaginary axes (but not w.r.t.\,\,the real axis, as for the local NLS equation \cite{ZS72}).
In order to be able to  deal with ``pairs'' of zeros in different complex half-planes, in addition to Assumption 1 we apply the following restriction on zeros of $a_j(k)$, $j=1,2$:

\noindent\textbf{Assumption 2.} \emph{We assume that the number of zeros of  $a_1(k)$ in $\mathbb{C}^+$ is equal to the number of zeros of $a_2(k)$  in  $\mathbb{C}^-$
	and that all these zeros are simple.}

Notice that the number of zeros is related to the winding number of the  function $a_1(k)a_2(k)$:
\begin{proposition}\label{wn}
	Suppose that $a_j(k)$, $j=1,2$ satisfy Assumption 1. Assume that $a_1(k)$ has $L_1$ zeros in $\mathbb{C}^+$ and $a_2(k)$ has $L_2$ zeros in $\mathbb{C}^-$.
	Then $\int_{-\infty}^{+\infty}d\arg a_1(k)a_2(k)\break =2\pi(L_1-L_2)$.
\end{proposition}
\begin{proof}
	See, e.g., Section 12.2 in \cite{Gah66}.
\end{proof}
\begin{remark}\label{BY}
	In view of Proposition \ref{wn}, Assumption 2 is equivalent to the
	Bar-Yaacov's winding number constraint \cite{BY85}:
	$$
	\int_{-\infty}^{+\infty}d\arg a_1(k)a_2(k)=0.
	$$
\end{remark}
Let us mark the
zeros of $a_j(k)$, $j=1,2$ by
(here we adopt the convention $\{(\cdot)_s\}_{s=1}^0=\emptyset$)
\begin{equation}\label{z-a}
\{\I\rho_{j,s}\}_{s=1}^{M_j}\cup\{\zeta_{j,s}, -\bar{\zeta}_{j,s}\}_{s=1}^{N_j},\quad
M_j,N_j\in\mathbb{N}\cup\{0\},
\end{equation}
where $(-1)^{j+1}\rho_{j,s}>0$ and $\Re\zeta_{j,s}<0$ (notice that $M_j=0$ for $\sigma=-1$, see Proposition \ref{defoc-zeros}).
Notice that in view of Assumption 2 we have
\begin{equation}\label{num_zer}
2 N_1+M_1=2 N_2+M_2=:L.
\end{equation}
Denote by
\begin{equation}\label{n-c}
\quad\{\gamma_{j,s}\}_{s=1}^{M_j}\cup\{\eta_{j,s}, \hat\eta_{j,s}\}_{s=1}^{N_j},\quad
M_j,N_j\in\mathbb{N}\cup\{0\},\,\,j=1,2,
\end{equation}
the associated norming constants,
which are defined as follows (see \eqref{det-rel}):
\begin{subequations}\label{norm-const}
	\begin{align}
	&\Phi_1^{[1]}(x,t,\I\rho_{1,s})
	=\gamma_{1,s}\Phi_2^{[2]}(x,t,\I\rho_{1,s}),
	\,\,s=1,\dots,M_1,\\
	&\Phi_2^{[1]}(x,t,\I \rho_{2,s})
	=\gamma_{2,s}\Phi_1^{[2]}(x,t,\I\rho_{2,s}),\,\,
	s=1,\dots,M_2,
	\end{align}
	and
	\begin{align}
	\nonumber
	&\Phi_1^{[1]}(x,t,\zeta_{1,s})
	=\eta_{1,s}\Phi_2^{[2]}(x,t,\zeta_{1,s}),&&
	\Phi_1^{[1]}(x,t,-\bar\zeta_{1,s})
	=\hat\eta_{1,s}\Phi_2^{[2]}(x,t,-\bar\zeta_{1,s}),\\
	& &&
	s=1,\dots,N_1,\\
	\nonumber
	&\Phi_2^{[1]}(x,t,\zeta_{2,s})
	=\eta_{2,s}\Phi_1^{[2]}(x,t,\zeta_{2,s}),&&
	\Phi_2^{[1]}(x,t,-\bar\zeta_{2,s})
	=\hat\eta_{2,s}\Phi_1^{[2]}(x,t,-\bar\zeta_{2,s}),\\
	& &&
	s=1,\dots,N_2.
	\end{align}
\end{subequations}
Since the columns $\Phi_{i}^{[j]}$, $i,j=1,2$ are nonzero solutions of the Lax pair \eqref{Lax} for all $x,t\in\mathbb{R}$, the norming constants do not depend on $x,t$.
From the symmetries \eqref{phi-sym-b} and \eqref{phi-sym-c} it follows that $|\gamma_{j,s}|=1$, $s=1,\dots,M_j$, $j=1,2$
and
$\hat\eta_{j,s}=\sigma\bar{\eta}_{j,s}^{-1}$, $s=1,\dots,N_j$, $j=1,2$.
\begin{remark}
	Setting $t=0$ in \eqref{det-rel} and \eqref{norm-const},  all the spectral functions and norming constants
	can be determined in terms of the  initial data only.
\end{remark}

Finally, let us recall the formal derivation  of the conservation laws for the Cauchy problem \eqref{iv-nnls} \cite{AMP, AMN}.
This will be useful in the sequel, in the  discussion of the conserved quantities of solutions  continued after blow-up  (see Remark \ref{concl-th-1}, Item (i) and Remark \ref{concl-th-2}, Item (i) below).

Let $\psi_1(x,t,k)=e^{\phi_1(x,t,k)}$ and $p(x,t,k)=\partial_x\phi_1(x,t,k)$.
Then from \eqref{psi-1-3} one concludes that $p$ satisfies the following Riccati equation
\begin{equation}\label{ricc}
q(x,t)\partial_x\left(\frac{p(x,t,k)}{q(x,t)}\right)
+p^2(x,t,k)-2\I kp(x,t,k)+\sigma q(x,t)\overline{q(-x,t)}=0.
\end{equation}
Notice that $\phi_1(x,t,k)\to 0$ as $k\to\infty$, $k\in\mathbb{C}^+$.
Substituting the expansion
\begin{equation}\label{p-n}
p(x,t,k)=\sum\limits_{n=1}^{\infty}
\frac{p_n(x,t)}{(2\I k)^n},
\end{equation}
into \eqref{ricc} we can find all $p_n$ recursively.
Particularly, the first three quantities have the form
\begin{equation}
\begin{split}
&p_1=\sigma q(x,t)\overline{q(-x,t)},\quad
p_2=\sigma q(x,t)
\partial_x\overline{q(-x,t)},\\
&p_3=q^2(x,t)\overline{q^2(-x,t)}
+\sigma q(x,t)\partial_{xx}\overline{q(-x,t)}.
\end{split}
\end{equation}
Taking into account that $\lim\limits_{x\to+\infty}\phi_1(x,t,k)=\ln a_1(k)$ and
$\lim\limits_{x\to-\infty}\phi_1(x,t,k)=0$ for $k\in\mathbb{C}^+$, we have
\begin{equation}\label{In-a1}
\ln a_1(k)=\sum\limits_{n=1}^{\infty}
\frac{I_n(t)}{(2\I k)^n},
\end{equation}
where $I_n(t)=\int_{-\infty}^{+\infty}p_n(x,t)\,dx$.
From (\ref{In-a1}) it follows that
$\{I_n\}_1^\infty$ do not depend on $t$
and thus constitute an infinite set of conservation laws for the NNLS equation with decaying boundary conditions.
In particular, the conservation laws of ``mass'' ($I_1$) and ``energy'' ($I_3$) have the form
\begin{align}
\label{cons-mass}
&\tilde M[q]=\int_{-\infty}^{+\infty}
q(x,t)\overline{q(-x,t)}\,dx=const,\\
\label{cons-energy}
&\tilde E[q]=\int_{-\infty}^{+\infty}
\left(\partial_xq(x,t)\partial_x\overline{q(-x,t)}
-\sigma q^2(x,t)\overline{q^2(-x,t)}\right)\,dx=const.
\end{align}
\subsection{Inverse scattering}\label{sec2.2}
We have defined the spectral data in terms of the given initial data $q_0(x)$.
This procedure involves solving a finite number of linear problems only.
Now we want to recover, at least formally, the solution $q(x,t)$ for all $t\in\mathbb{R}$ from the (known) spectral data.
Define a matrix-valued function $M$ as follows:
\begin{equation}
\label{M}
M(x,t,k)=\begin{cases}
\left(\frac{\Psi_1^{[1]}(x,t,k)}{a_{1}(k)},\Psi_2^{[2]}(x,t,k)\right), & \Im k\geq 0, \\
\left(\Psi_2^{[1]}(x,t,k),\frac{\Psi_1^{[2]}(x,t,k)}{a_{2}(k)}\right), & \Im k\leq 0.
\end{cases}
\end{equation}
Then the scattering relation \eqref{S-def} together with $\Phi_j(x,t,k)=\Psi_j(x,t,k)e^{(-\I kx-2\I k^{2}t)\sigma_{3}}$ imply the following jump condition for $M(\cdot,\cdot,k)$
across $k\in {\mathbb R}$ (the real line is oriented from left to right):
\begin{equation}
\label{jump-M}
M_{+}(x,t,k)=M_{-}(x,t,k)J(x,t,k),\quad k\in {\mathbb R},
\end{equation}
where
$M_\pm$ denotes the limiting values of $M$ as $k$ approaches $\mathbb R$ from the ``$+$'' (left) or the  ``$-$'' (right) side and
\begin{equation}
\label{ist4.3}
J(x,t,k)=
\begin{pmatrix}
1+\sigma r_{1}(k)r_{2}(k)& \sigma r_{2}(k)e^{-2\I kx-4\I k^2t}\\
r_1(k)e^{2\I kx+4\I k^2t}& 1
\end{pmatrix},\quad k\in {\mathbb R},
\end{equation}
with the reflection coefficients $r_j$, $j=1,2$ defined as follows:
\begin{equation}\label{r-12}
r_1(k)=\frac{b(k)}{a_1(k)},\quad
r_2(k)=\frac{\overline{b(-{k})}}{a_2(k)},\quad k\in\mathbb{R}.
\end{equation}
Additionally,
$M(x,t,k)$ satisfies
the normalization condition at infinity:
\begin{equation}
\label{norm}
M(x,t,k)\rightarrow I, \quad k\rightarrow\infty,
\end{equation}
and the residue conditions at zeros of $a_j(k)$, $j=1,2$, if any (see \eqref{z-a} -- \eqref{n-c} and recall that $\hat\eta_{j,s}=\sigma\bar\eta_{j,s}^{-1}$):
\begin{itemize}
	\item in the upper half plane,
	\begin{subequations}\label{res-upper}
		\begin{align}
		&\underset{k=\I \rho_{1,s}}{\operatorname{Res}}
		M^{[1]}(x,t,k)=\frac{\gamma_{1,s}}{\dot{a}_1(\I\rho_{1,s})}
		e^{-2\rho_{1,s}x-4\I\rho_{1,s}^2t}
		M^{[2]}(x,t,\I\rho_{1,s}),\,\,s=1,\dots,M_1,\\
		&\underset{k=\zeta_{1,s}}{\operatorname{Res}}
		M^{[1]}(x,t,k)=\frac{\eta_{1,s}}{\dot{a}_1(\zeta_{1,s})}
		e^{2\I\zeta_{1,s}x+4\I\zeta_{1,s}^2t}
		M^{[2]}(x,t,\zeta_{1,s}),
		\quad \,\,\,\,\,s=1,\dots,N_1,\\
		&\underset{k=-\bar\zeta_{1,s}}{\operatorname{Res}}
		M^{[1]}(x,t,k)=
		\frac{\sigma e^{-2\I\bar\zeta_{1,s}x+4\I\bar\zeta_{1,s}^2t}}
		{\bar\eta_{1,s}\dot{a}_1(-\bar\zeta_{1,s})}
		M^{[2]}(x,t,-\bar\zeta_{1,s}),
		\qquad \,\,\,s=1,\dots,N_1,
		\end{align}
	\end{subequations}
	\item in the lower half plane,
	\begin{subequations}\label{res-lower}
		\begin{align}
		&\underset{k=\I \rho_{2,s}}{\operatorname{Res}}
		M^{[2]}(x,t,k)=\frac{\gamma_{2,s}}{\dot{a}_2(\I\rho_{2,s})}
		e^{2\rho_{2,s}x+4\I\rho_{2,s}^2t}
		M^{[1]}(x,t,\I\rho_{2,s}),\,\,\,\,\,s=1,\dots,M_2,\\
		&\underset{k=\zeta_{2,s}}{\operatorname{Res}}
		M^{[2]}(x,t,k)=\frac{\eta_{2,s}}{\dot{a}_2(\zeta_{2,s})}
		e^{-2\I\zeta_{2,s}x-4\I\zeta_{2,s}^2t}
		M^{[1]}(x,t,\zeta_{2,s}),\quad s=1,\dots,N_2,\\
		&\underset{k=-\bar\zeta_{2,s}}{\operatorname{Res}}
		M^{[2]}(x,t,k)=
		\frac{\sigma e^{2\I\bar\zeta_{2,s}x-4\I\bar\zeta_{2,s}^2t}}
		{\bar\eta_{2,s}\dot{a}_2(-\bar\zeta_{2,s})}
		M^{[1]}(x,t,-\bar\zeta_{2,s}),
		\qquad \,\,\,\,\, s=1,\dots,N_2.
		\end{align}
	\end{subequations}
\end{itemize}
Moreover, $M$ satisfies the following symmetry relation, which can be derived from the symmetries of $a_j$ described in Item 3 in Proposition \ref{a_j,b} (see Lemma 1 and Remark 3 in \cite{RS19} for details)
\begin{equation}\label{M-sym}
M(x,t,k)=\begin{cases}
\Lambda \overline{M(-x,t,-\bar k)} \Lambda^{-1}
\begin{pmatrix}
\frac{1}{a_1(k)} & 0 \\ 0 & a_1(k)
\end{pmatrix}, & k\in\overline{\mathbb{C}^+}, \\
\Lambda \overline{M(-x,t,-\bar k)} \Lambda^{-1}
\begin{pmatrix}
a_2(k) & 0 \\ 0 & \frac{1}{a_2(k)}
\end{pmatrix}, & k\in\overline{\mathbb{C}^-}.
\end{cases}
\end{equation}
Particularly, \eqref{M-sym} implies that
\begin{equation}\label{M-sym-as}
M(x,t,k)=\Lambda\overline{M(-x,t,-\bar k)}\Lambda^{-1}(I+\text{daig}\{\ord(k^{-1}),\ord(k^{-1})\}),\quad k\to\infty,
\end{equation}
where $\text{diag}\{a,b\}=
\left(
\begin{smallmatrix}
a&0\\
0& b
\end{smallmatrix}
\right)
$.

Given $M(x,t,k)$,
the solution
$q(x,t)$ of the Cauchy problem (\ref{iv-nnls})
can be expressed in terms of $M(x,t,k)$ as follows:
\begin{equation}
\label{q(x,t)}
q(x,t)=2\I\lim_{k\rightarrow\infty}k\, M_{12}(x,t,k),
\end{equation}
or, in view of \eqref{M-sym-as},
\begin{equation}
\label{q(-x,t)}
q(x,t)=-2\I\sigma\lim_{k\rightarrow\infty} \overline{k\,M_{21}(-x,t,k)}.
\end{equation}

The relations above can be interpreted as a  procedure
of reconstructing $q(x,t)$
in terms of the solution of
an associated RH problem
with data
determined by $q(x,0)$.
This RH problem is referred to as \textit{basic} and is formulated as follows:


\textbf{Basic RH problem:} find piecewise meromorphic $2\times2$ valued function $M(x,t,k)$ which satisfies the jump condition \eqref{jump-M}, the residue conditions \eqref{res-upper} and \eqref{res-lower}, the normalization condition \eqref{norm} and the symmetry \eqref{M-sym}.
Here the spectral functions $a_j(k)$, $r_j(k)$ and the norming constants $\gamma_{j,s}$, $\nu_{j,s}$ are determined by $q_0(x)$ through \eqref{det-rel} and \eqref{norm-const} with $t=0$ and \eqref{r-12}.
\begin{remark}
	The spectral functions $a_j$ can be expressed in terms of their zeros and the reflection coefficients $r_j$ via the  trace formulas (cf.\,\,Section 8 in \cite{AMN}):
	\begin{subequations}
		\begin{align}
		\nonumber
		a_1(k)=&\prod\limits_{s=1}^{M_1}
		\frac{k-\I\rho_{1,s}}{k-\I\rho_{2,s}}
		\prod\limits_{s=1}^{N_1}
		\frac{(k-\zeta_{1,s})(k+\bar\zeta_{1,s})}
		{(k-\zeta_{2,s})(k+\bar\zeta_{2,s})}\\
		&\times
		\exp\left\{
		\frac{-1}{2\pi\I}\int_{-\infty}^{+\infty}
		\frac{1+\sigma r_1(\zeta)r_2(\zeta)}{\zeta-k}\,d\zeta
		\right\},
		\quad k\in\overline{\mathbb{C}^+},\\
		\nonumber
		a_2(k)=&\prod\limits_{s=1}^{M_2}
		\frac{k-\I\rho_{2,s}}{k-\I\rho_{1,s}}
		\prod\limits_{s=1}^{N_2}
		\frac{(k-\zeta_{2,s})(k+\bar\zeta_{2,s})}
		{(k-\zeta_{1,s})(k+\bar\zeta_{1,s})}\\
		&\times\exp\left\{
		\frac{1}{2\pi\I}\int_{-\infty}^{+\infty}
		\frac{1+\sigma r_1(\zeta)r_2(\zeta)}{\zeta-k}\,d\zeta
		\right\},
		\quad k\in\overline{\mathbb{C}^-}.
		\end{align}
	\end{subequations}	
\end{remark}
Since $\det J(x,t,k)=1$ and $\det M(x,t,k)\to1$ as $k\to\infty$, it follows, by the Liouville theorem,
that
$\det M(x,t,k)=1$ for all $x,t\in\mathbb{R}$, $k\in\mathbb{C}$.
Moreover, the solution of the basic RH problem is unique, if it exists.
Indeed, if $M$ and $\hat M$ are two solutions, then $M\hat M^{-1}$ has neither jump across $\mathbb{R}$ nor singularities.
Since $M\hat M^{-1}\to I$ as $k\to\infty$,
we have (again
by the Liouville theorem) that $M\hat M^{-1}=I$.

To sum up, we have constructed the direct scattering map $\mathcal{R}$
\begin{equation}
\mathcal{R}(q_0(\cdot))=\{r_j(\cdot);
\mathcal{Z},\mathcal{N}(0)\},
\end{equation}
and the  map $\mathcal{I}$
\begin{equation}
\mathcal{I}(r_j(\cdot)e^{(-1)^{j+1}4\I(\cdot)^2t};
\mathcal{Z}, \mathcal{N}(t))= q(\cdot,t),
\end{equation}
where $\mathcal{Z}=\bigcup\limits_{j=1}^{2}
\{\I\rho_{j,s}\}_{s=1}^{M_j}
\cup\{\zeta_{j,s}, -\bar{\zeta}_{j,s}\}_{s=1}^{N_j}$ and
$\mathcal{N}(t)=\bigcup\limits_{j=1}^{2}
\{\gamma_{j,s}e^{(-1)^j4\I\rho_{j,s}^2t}\}_{s=1}^{M_j}
\cup\{\eta_{j,s}e^{(-1)^{j+1}4\I\zeta_{j,s}^2t}\}_{s=1}^{N_j}$,
which, due to the uniqueness of the solution of the RH problem, can be considered as a formal inverse to
$\mathcal{R}$.
\begin{remark}
	Using the determinant relation (see Proposition \ref{a_j,b}, Item 4) and taking into account that $a_j$ do not have zeros on $\mathbb{R}$ (see Assumption 1), we have
	\begin{equation}\label{r-a}
	1+\sigma r_1(k) r_2(k)=\frac{1}{a_1(k)a_2(k)},\quad k\in {\mathbb R}.
	\end{equation}
	Therefore, the winding number constraint (see Remark \ref{BY}) in terms of $r_j$ reads as (cf.\,\,\cite{BY85, BC84, RS19, Z98})
	\begin{equation}
	\int_{-\infty}^{+\infty}d\arg(1+\sigma r_1(k)r_2(k))=0.
	\end{equation}
\end{remark}

\section{Smooth solutions with blow-up points}\label{sm-iv}
In this Section we assume that $q_0(x)\in\mathcal{S}(\mathbb{R})$ (Schwartz space functions).
Then the reflection coefficients also belong to the Schwartz space \cite{BC84, FT} (see also Proposition \ref{r-S}).

Solving the inverse problem is more involved.
Consider the basic RH problem without residue conditions (the case with zeros can be reduced to this one, see \cite{FT, Z98} and Theorems \ref{th-class-L-sol} and \ref{th-H11-L-sol} below). Let us reduce it to a singular integral equation \cite{BC84} (see also, e.g., \cite{DZ03, Z98}).
The jump matrix $J$ admits the following factorization:
\begin{equation}\label{upp-low}
J(x,t,k)=(I-w^{-}(x,t,k))^{-1}(I+w^{+}(x,t,k)),
\end{equation}
where
$w^{-}(x,t,k)=
\left(
\begin{smallmatrix}
0 & \sigma r_2(k)e^{-2\I kx-4\I k^2t}\\
0 & 0
\end{smallmatrix}
\right)
$ and
$
w^{+}(x,t,k)=
\left(
\begin{smallmatrix}
0 & 0\\
r_1(k)e^{2\I kx+4\I k^2t} & 0
\end{smallmatrix}
\right)
$.
Then it can be shown that
\begin{equation}\label{M-int-eq}
M(x,t,k)=I+\frac{1}{2\pi\I}\int_{-\infty}^{+\infty}
\frac{\mu(x,t,\zeta)(w^+(x,t,\zeta)+w^-(x,t,\zeta))}{\zeta-k}\,d\zeta,\quad k\in\mathbb{C}\setminus\mathbb{R},
\end{equation}
where $\mu$ solves the linear singular equation
\begin{equation}\label{mu-int}
\mu(x,t,k)=I+\mathcal{C}_w(\mu(x,t,k)),\quad k\in\mathbb{R}.
\end{equation}
Here $\mathcal{C}_w(\mu )=
\mathcal{C}_+(\mu w^-)
+\mathcal{C}_-(\mu w^+)$ with
\begin{equation}
\mathcal{C}_\pm f(k)=
\lim\limits_{\varepsilon\downarrow 0}
\frac{1}{2\pi\I}\int_{-\infty}^{+\infty}
\frac{f(\zeta)}{\zeta-(k\pm\I\varepsilon)}\,d\zeta,\quad k\in\mathbb{R}.
\end{equation}

Introducing $\mu^{\#}=\mu-I$, equation \eqref{mu-int} reduces to \eqref{int-mu} below, which can be considered as
an equation in  $L^2(\mathbb{R})$:
\begin{equation}\label{int-mu}
(I-\mathcal{C}_w)\mu^{\#}(x,t,k)=\mathcal{C}_w(I),\quad k\in\mathbb{R}.
\end{equation}
Therefore, the RH problem is solvable (in the $L^2$ sense) if there exists a bounded inverse $(I-\mathcal{C}_w)^{-1}$.
This takes place, in particular, when the jump matrix satisfies  two conditions: (i) $J(x,t,k)+J^{\dagger}(x,t,k)$ is positive definite for $k\in\mathbb{R}$ and (ii) $J(x,t,\bar{k})=J^{\dagger}(x,t,k)$ on any parts of the contour except $\mathbb{R}$, if any \cite{Z89-1} (here $J^{\dagger}$ is the Hermitian conjugate of $J$).
Notice that in our problem, the reflection coefficients $r_j$ do not satisfy  symmetries which would imply that $J(x,t,k)+J^{\dagger}(x,t,k)$ is positive definite for $k\in\mathbb{R}$ (cf., e.g., the defocusing NLS equation \cite{DZ03} or the derivative NLS equation \cite{JLPS20, LYF22, PSS17, PS18}).
Therefore, we will use another sufficient condition for solvability,
which is based on a small norm
assumption. Namely,  the operator $(I-\mathcal{C}_w)$ is invertible in $L^2$ if $\|\mathcal{C}_w\|_{L^2\to L^2}<1$.
Taking into account that $\|\mathcal{C}_\pm\|_{L^2\to L^2}=1$, we have
\begin{equation}
\|\mathcal{C}_w h\|_{L^2(\mathbb{R})}\leq\|h\|_{L^2(\mathbb{R})}
\max\{
\|w^+(x,t,\cdot)\|_{L^\infty(\mathbb{R})},
\|w^-(x,t,\cdot)\|_{L^\infty(\mathbb{R})}
\}.
\end{equation}
This implies that if $\|r_j(k)\|_{L^\infty(\mathbb{R})}<1$, $j=1,2$, then  \eqref{int-mu} has a solution in $L^2$.

\subsection{Classical solutions}
In this subsection we obtain the Schwartz global solution of the initial value problem \eqref{iv-nnls} with small, in a sense, initial profile (see Item 2 in Theorem \ref{th-class-glob} and Proposition \ref{prop-suff} below).
At first we specify what we mean in this paper by a classical solution of  problem \eqref{iv-nnls}.
\begin{definition}(Classical Schwartz solution)
	We say that $q(x,t)$ is a classical Schwartz solution of the Cauchy problem \eqref{iv-nnls} for $t\in[-T_1,T_2]$, $T_j>0$, with $q_0(x)\in\mathcal{S}(\mathbb{R})$ if
	\begin{enumerate}[(i)]
		\item $q(x,t)$ satisfies equation \eqref{nnls} for all $x\in\mathbb{R}$, $t\in[-T_1,T_2]$,
		\item $q(x,0)=q_0(x)$ for all $x\in\mathbb{R}$,
		\item $\partial_t^nq(\cdot,t)\in\mathcal{S}(\mathbb{R})$ for all $t\in[-T_1,T_2]$, $n\in\mathbb{N}\cup\{0\}$.
	\end{enumerate}
\end{definition}

\begin{theorem}\label{th-class-glob}
	(Classical Schwartz solution: no solitons).
	Suppose that the initial data $q_0(x)\in\mathcal{S}(\mathbb{R})$ is such that:
	\begin{enumerate}
		\item The spectral function $a_1(k)$ does not have zeros in ${\mathbb C}^+$ and the
		spectral function $a_2(k)$ does not have zeros in ${\mathbb C}^-$.
		\item $\|r_j(k)\|_{L^\infty(\mathbb{R})}<1$, $j=1,2$.
	\end{enumerate}
	Then the Cauchy problem \eqref{iv-nnls} has the unique global classical Schwartz solution $q(x,t)$, which can be obtained by \eqref{q(x,t)} in terms of the solution of the basic RH problem.
\end{theorem}
\begin{proof}
	First, since $\|r_j(k)\|_{L^\infty(\mathbb{R})}<1$, there exists (and unique) $M(x,t,k)$ which satisfies the corresponding basic RH problem.
	Since $q_0(x)\in\mathcal{S}(\mathbb{R})$, from Proposition \ref{r-S} we conclude that $r_j(k)\in\mathcal{S}(\mathbb{R})$.
	Then, Proposition \ref{r-j-Schwartz} implies that $\partial_t^nq(\cdot,t)\in\mathcal{S}(\mathbb{R})$.
	Moreover, in view of Corollaries \ref{x-deriv-M} and \ref{t-deriv-M} we can differentiate $M(x,t,k)$ w.r.t.\,\,$x$ and $t$, so
	from Propositions \ref{inv-x} and \ref{inv-t} we conclude that $M(x,t,k)e^{(-\I kx-2\I k^2t)\sigma_3}$ satisfies the Lax pair \eqref{Lax} with $q(x,t)$ determined by
	\eqref{q(x,t)} and  satisfying the NNLS equation.
	
	
	
	Let us prove uniqueness. The compatibility condition $\Phi_{xt}=\Phi_{tx}$ of the system \eqref{Lax} is valid for the two Schwartz solutions $q$ and $\tilde{q}$, which implies that they both can be obtained by \eqref{q(x,t)} via the unique solution of the same RH problem specified by $q(x,0)=\tilde{q}(x,0)=q_0(x)$.
\end{proof}
\begin{remark}
	The long-time asymptotic behavior
	of the global solution $q(x,t)$ along the rays $x/t=const$ under assumptions of  Theorem \ref{th-class-glob} was obtained in \cite{RS19} by applying the nonlinear steepest decent method \cite{DZ}.
\end{remark}
\begin{proposition}\label{prop-suff}
	(Sufficient condition for Items 1,2 in Theorem \ref{th-class-glob}).
	If $q_0(x)\in\mathcal{S}(\mathbb{R})$ is such that
	\begin{equation}\label{suff-Bessel}
	\frac{1}{2}(\|q_0\|_{L^1(\mathbb{R})}+1)
	I_0(2\|q_0\|_{L^1(\mathbb{R})})
	<1,
	\end{equation}
	where $I_0(\cdot)$ is the modified Bessel function, then assumprions 1 and 2 in Theorem \ref{th-class-glob} are fulfilled.
\end{proposition}
Condition \eqref{suff-Bessel} holds, particularly, if
$\|q_0\|_{L^1(\mathbb{R})}\leq 0.532$,
but is violated for
$\|q_0\|_{L^1(\mathbb{R})}\geq0.533$.
\begin{proof}
	From \eqref{psi-1-3} we conclude that $\psi_1$ satisfies the Volterra integral equation (cf.\,\,\break Section 1.3 in \cite{AS})
	\begin{equation}\label{psi-1-V}
	\psi_1(x,0,k)=1-\sigma\int_{-\infty}^{x}
	\overline{q_{0}(-y)}
	\int_y^x e^{2\I k(z-y)}
	q_0(z)\,dz\,\psi_1(y,0,k)\,dy.
	\end{equation}
	Using the Neumann series for $\psi_1$, we can write it in the form $\psi_1(x,0,k)=1+F_1(x,k)$, with
	\begin{equation}\label{N-ser-F-1}
	\begin{split}
	|F_1(x,k)|&\leq
	Q_0(x)R_0(x)+\frac{1}{(2!)^2}Q_0^2(x)R_0^2(x)+\frac{1}{(3!)^2}Q_0^3(x)R_0^3(x)+\ldots\\
	&=
	I_0(2\sqrt{Q_0(x)R_0(x)})-1,
	\end{split}
	\end{equation}
	for $k\in\overline{\mathbb{C}^+}$,
	where $Q_0(x)=\int_{-\infty}^{x}|q_0(y)|\,dy$ and $R_0(x)=\int_{-\infty}^{x}|q_{\,0}(-y)|\,dy$.
	Since $R_0(\infty)=Q_0(\infty)=\|q_0\|_{L^1}$, it follows that if $I_0(2\|q_0\|_{L^1})<2$, then  $a_1(k)=1+F_1(\infty,k)$ for $\Im k\in\overline{\mathbb{C}^{+}}$ with $|F_1(\infty,k)|<1$,
	where $F_1(\infty,k)
	=\lim\limits_{x\to+\infty}F_1(x,k)$.
	The latter implies that $a_1$ does not have zeros in $\overline{\mathbb{C}^{+}}$.
	The same is true for $a_2(k)$ in $\overline{\mathbb{C}^{-}}$.
	
	Since $I_0(\cdot)$ is monotone for positive arguments, from the second equation in \eqref{psi-1-3} we have
	\begin{equation}
	|\psi_3(x,k)|\leq I_0(2\sqrt{Q_0(x)R_0(x)})R_0(x).
	\end{equation}
	Moreover, $|a_1(k)|\geq 2-I_0(2\|q_0\|_{L^1})$, so for $r_1$ we have the estimate
	\begin{equation}
	|r_1(k)|\leq\frac{\|q_0\|_{L^1}I_0(2\|q_0\|_{L^1})}
	{2-I_0(2\|q_0\|_{L^1})},\quad
	k\in\mathbb{R}.
	\end{equation}
	This implies that if \eqref{suff-Bessel} holds, then $\|r_1(k)\|_{L^{\infty}}<1$.
	The same is true for $r_2(k)$.
\end{proof}
\subsection{Smooth solutions with blow-up points}

Allowing the spectral functions $a_j(k)$ to have zeros leads to RH problems with nontrivial residue conditions, which eventually lead to solutions of the Cauchy problem \eqref{iv-nnls} that may blow-up
at some finite $x$ and $t$ (which, in particular, is the case of the
one soliton solution \eqref{one-soliton}).
In Theorems \ref{th-class-one-sol} and \ref{th-class-L-sol} below we show that  $q(x,t)$ given by \eqref{q(x,t)} is the solution of the Cauchy problem \eqref{iv-nnls} in the whole $(x,t)$ plane except the points where it can blow up.

Let us give a precise definition of a smooth solution with singularities.
\begin{definition}\label{blow-up-S}
	(Blowing-up Schwartz solution)
	We say that $q(x,t)$ is a blowing-up Schwartz solution of the Cauchy problem \eqref{iv-nnls} with
	$q_0(x)\in\mathcal{S}(\mathbb{R})$
	and with a closed set of
	blow-up points $\mathcal{B}$, which lies either in the band $\{(x,t)\in\mathbb{R}^2:|x|<R\}$ for some $R>0$ or in the sector
	$\{(x,t)\in\mathbb{R}^2:|t|>\alpha|x|-C\}$, for some $\alpha, C>0$,
	if
	\begin{enumerate}[(i)]
		\item $q(x,t)$ satisfies equation \eqref{nnls} for all $(x,t)\in\mathbb{R}^2\setminus\mathcal{B}$ while $q(x,t)=\infty$ for all $(x,t)\in\mathcal{B}$,
		\item $q(x,0)=q_0(x)$ for all $x\in\mathbb{R}$,
		\item for all
		$(x,t)\in\mathbb{R}^2\setminus\mathcal{B}$,
		the derivatives
		$\partial_x^s\partial_t^nq(x,t)$ exist for all $s,n\in\mathbb{N}\cup\{0\}$; moreover,
		$|x|^m\partial_x^s\partial_t^nq(x,t)\to 0$ as $x\to\pm\infty$ for all $s,n,m\in\mathbb{N}\cup\{0\}$ and all fixed $t\in\mathbb{R}$.
	\end{enumerate}
\end{definition}

To fix ideas, let us consider first
the case of one pair of zeros $k=\I\rho_{j1}$, $j=1,2$ (observe that due to Proposition \ref{defoc-zeros},
this is possible in the case $\sigma=1$ only).
Notice that in order to establish smoothness at $(x,t)\in\mathbb{R}\setminus\mathcal{B}$, we should differentiate the RH problem w.r.t. both $x$ and $t$, cf. \cite{DZ03,JLPS20}.
\begin{theorem}\label{th-class-one-sol}
	(Blowing-up Schwartz solution: one soliton)
	Consider the Cauchy problem \eqref{iv-nnls} with $\sigma=1$ and with initial data $q_0(x)\in\mathcal{S}(\mathbb{R})$ satisfying the following properties:
	\begin{enumerate}
		\item $a_1(k)$ has one simple zero $k=\I\rho_{1,1}$, $\rho_{1,1}>0$
		and $a_2(k)$ has one simple zero $k=\I\rho_{2,1}$, $\rho_{2,1}<0$.
		\item $\|r_j^{\mathrm{reg}}(k)\|_{L^\infty(\mathbb{R})}<1$, $j=1,2$, where
		$r_1^{\mathrm{reg}}(k):=
		\frac{k-\I\rho_{1,1}}{k-\I\rho_{2,1}}r_1(k)$ and
		$r_2^{\mathrm{reg}}(k):=
		\frac{k-\I\rho_{2,1}}{k-\I\rho_{1,1}}r_2(k)$.
	\end{enumerate}
	Let the 2-vectors
	$g=\left(
	\begin{smallmatrix}
	g_1\\
	g_2
	\end{smallmatrix}
	\right)$
	and
	$h=\left(
	\begin{smallmatrix}
	h_1\\
	h_2
	\end{smallmatrix}
	\right)$ be defined by
	\begin{subequations}\label{g-j-h-j}
		\begin{align}
		\nonumber
		g(x,t)=&\,\I(\rho_{1,1}-\rho_{2,1})
		(M^{\mathrm{reg}})^{[1]}(x,t,\I\rho_{1,1})\\
		&-\frac{\gamma_{1,1}}{\dot{a}_1(\I\rho_{1,1})}
		e^{-2\rho_{1,1}x-4\I\rho_{1,1}^2t}
		(M^{\mathrm{reg}})^{[2]}(x,t,\I\rho_{1,1}),\\
		\nonumber
		h(x,t)=&\,\I(\rho_{1,1}-\rho_{2,1})
		(M^{\mathrm{reg}})^{[2]}(x,t,\I\rho_{2,1})\\
		&+\frac{\gamma_{2,1}}{\dot{a}_2(\I\rho_{2,1})}
		e^{2\rho_{2,1}x+4\I\rho_{2,1}^2t}
		(M^{\mathrm{reg}})^{[1]}(x,t,\I\rho_{2,1}),
		\end{align}
	\end{subequations}
	where $M^{\mathrm{reg}}$ solves the RH problem \eqref{M-reg-RH}, and let $\mathcal{B}$ be defined as follows:
	\begin{equation}\label{B-one-sol}
	\mathcal{B}:=\{(x,t)\in\mathbb{R}^2:
	g_1(x,t)h_2(x,t)-g_2(x,t)h_1(x,t)=0\}.
	\end{equation}
	Then
	\begin{enumerate}[(i)]
		\item
		$\mathcal{B}$ is a closed set, which lies in a band $\{(x,t)\in\mathbb{R}^2:|x|<R\}$ for some $R>0$;
		\item
		$q(x,t)$ defined by \eqref{q(x,t)} via the solution of the  basic RH problem is a blowing-up Schwartz solution, with blow-up points in $\mathcal{B}$.
	\end{enumerate}
	Moreover, $q(x,t)$ converges to the one-soliton solution  as $x$ and $t$ become large:
	\begin{equation}\label{one-sol-gen-th}
	\begin{split}
	q(x,t)=
	&\frac{-2\I(\rho_{1,1}-\rho_{2,1})^2\dot{a}_1(\I\rho_{1,1})}
	{\gamma_{2,1}^{-1}\dot{a}_1(\I\rho_{1,1})
		\dot{a}_2(\I\rho_{2,1})
		(\rho_{1,1}-\rho_{2,1})^2
		e^{-2\rho_{2,1}x-4\I\rho_{2,1}^2t}
		-\gamma_{1,1}e^{-2\rho_{1,1}x-4\I\rho_{1,1}^2t}}\\
	&+\osmall(1),
	\end{split}
	\end{equation}
	as $|x|^2+|t|^2 \to\infty$ avoiding, if $\rho_{1,1}\neq-\rho_{2,1}$, all $(x^\prime, t_s)$ with
	\begin{subequations}\label{sing-points}
		\begin{align}
		&x^\prime=\frac{1}{2(\rho_{2,1}-\rho_{1,1})}
		\ln|\dot{a}_1(\I\rho_{1,1})
		\dot{a}_2(\I\rho_{2,1})
		(\rho_{1,1}-\rho_{2,1})^2|,\\
		& t_s=\frac{\arg\gamma_{1,1}
			-\arg(\gamma_{2,1}^{-1}\dot{a}_1(\I\rho_{1,1})
			\dot{a}_2(\I\rho_{2,1})) + 2\pi s}
		{4(\rho_{1,1}^2-\rho_{2,1}^2)},\quad s\in\mathbb{Z}.
		\end{align}
	\end{subequations}
	
\end{theorem}
\begin{proof}
	Let us transform the basic RH problem with residue conditions to a regular RH problem.
	Consider the $2\times2$ Blaschke-Potapov factor \cite{FT}
	$$
	B(x,t,k)=I+\frac{\I\rho_{2,1}-\I\rho_{1,1}}{k-\I\rho_{2,1}}P(x,t),
	$$
	where $P$ is a projector ($P^2(x,t)=P(x,t)$) to be specified,
	and define  $M^{\mathrm{reg}}$ by
	\begin{equation}\label{M-reg}
	M^{\mathrm{reg}}(x,t,k)=B(x,t,k)M(x,t,k)
	\begin{pmatrix}
	1& 0\\
	0& \frac{k-\I\rho_{2,1}}{k-\I\rho_{1,1}}
	\end{pmatrix}.
	\end{equation}
	Direct calculations show that if the projector $P(x,t)$ has the following kernel and image:
	\begin{equation}\label{ker-im}
	\ker P(x,t)=\mathrm{lin}_\mathbb{C}\{g(x,t)\}, \quad
	\Im P(x,t)=\mathrm{lin}_\mathbb{C}\{h(x,t)\},
	\end{equation}
	where the 2-vectors
	$g$ and $h$ are defined by \eqref{g-j-h-j},
	then $M^{\mathrm{reg}}$ satisfies the following
	regular RH problem (solvable due to condition $2$ in  Theorem):
	\begin{subequations}\label{M-reg-RH}
		\begin{align}
		&M^{\mathrm{reg}}_+(x,t,k)=M^{\mathrm{reg}}_-(x,t,k)
		J^{\mathrm{reg}}(x,t,k),&&k\in\mathbb{R},\\
		&M^{\mathrm{reg}}(x,t,k)\to I&&k\to\infty.
		\end{align}
	\end{subequations}
	Here
	\begin{equation}
	J^{\mathrm{reg}}(x,t,k)=
	\begin{pmatrix}
	1+r_1^{\mathrm{reg}}(k)r_{2}^{\mathrm{reg}}(k)
	& r_2^{\mathrm{reg}}(k)e^{-2\I kx-4\I k^2t}\\
	r_1^{\mathrm{reg}}(k)e^{2\I kx+4\I k^2t}& 1
	\end{pmatrix},
	\end{equation}
	with $r_1^{\mathrm{reg}}(k)=
	\frac{k-\I\rho_{1,1}}{k-\I\rho_{2,1}}r_1(k)$ and
	$r_2^{\mathrm{reg}}(k)=
	\frac{k-\I\rho_{2,1}}{k-\I\rho_{1,1}}r_2(k)$.
	Notice that conditions \eqref{ker-im} uniquely determine $P$ (we drop arguments $x$, $t$ for simplicity):
	\begin{equation}\label{P-i-j}
	P_{12}=\frac{g_1h_1}{g_1h_2-g_2h_1},\,
	P_{21}=\frac{-g_2h_2}{g_1h_2-g_2h_1},\,
	P_{11}=-\frac{P_{12}g_2}{g_1},\,P_{22}=-\frac{P_{21}g_1}{g_2}.
	\end{equation}
	From \eqref{q(x,t)} and \eqref{M-reg} we conclude that  $q(x,t)$ can be determined in terms of $M^{\mathrm{reg}}$ as follows:
	\begin{equation}\label{q-qsol-qreg}
	q(x,t)=\tilde{q}^{\mathrm{sol}}_1(x,t)
	+q^{\mathrm{reg}}(x,t),
	\end{equation}
	where
	\begin{subequations}\label{q-1-sol-qreg}
		\begin{align}
		\label{q-1-sol}
		&\tilde{q}^{\,\mathrm{sol}}_1(x,t)=
		-2(\rho_{1,1}-\rho_{2,1})P_{12}(x,t),\\
		&q^{\mathrm{reg}}(x,t)=
		2\I\lim\limits_{k\to\infty}kM_{12}^{\mathrm{reg}}(x,t,k).
		\end{align}
	\end{subequations}
	
	Now let us prove that $q(x,t)$
	is a blowing-up Schwartz solution with blow-up points in $\mathcal{B}$.
	To this end we investigate the properties of the functions
	$\tilde{q}^{\,\mathrm{sol}}_1$ and $q^{\mathrm{reg}}(x,t)$.
	Concerning the latter, observe that arguing similarly as in Theorem \ref{th-class-glob} we have that
	$\partial_t^nq^{\mathrm{reg}}(\cdot,t)
	\in\mathcal{S}(\mathbb{R})$ for all $n\in\mathbb{N}\cup\{0\}$.
	
	To study $\tilde{q}^{\,\mathrm{sol}}_1$ we first notice that Proposition \ref{M-fixed-k}, Item 2 for $M^{\mathrm{reg}}$ implies that the function $g_1h_2-g_2h_1$ (see \eqref{g-j-h-j}) is smooth and therefore the set $\mathcal{B}$ of its zeros is closed.
	Moreover, from \eqref{g-j-h-j}, \eqref{P-i-j} and \eqref{q-1-sol} we conclude that
	$\partial_x^s\partial_t^n
	\tilde{q}^{\,\mathrm{sol}}_1(x,t)$  exists
	for all $(x,t)\in\mathbb{R}^2\setminus\mathcal{B}$ and $s,n\in\mathbb{N}\cup\{0\}$.
	Then we observe that Proposition \ref{M-fixed-k}, Item 2 for $M^{\mathrm{reg}}$ and the definition of $P_{12}$
	imply that
	$|x|^m\partial_x^s\partial_t^n
	\tilde{q}^{\,\mathrm{sol}}_1(x,t)\to 0$ as $x\to\pm\infty$ for all $m,s,n\in\mathbb{N}\cup\{0\}$ and fixed $t\in\mathbb{R}$.
	
	After having verified the above properties of
	$\tilde{q}^{\,\mathrm{sol}}_1$ and
	$q^{\mathrm{reg}}$,
	Propositions \ref{inv-x} and \ref{inv-t} imply that $q(x,t)$ satisfies equation \eqref{nnls} with $\sigma=1$
	for all $(x,t)\in\mathbb{R}^2\setminus\mathcal{B}$.
	To show that $\mathcal{B}$ lies in the band
	$\{(x,t)\in\mathbb{R}^2:|x|<R\}$ it is enough to
	prove that \eqref{q-qsol-qreg} implies \eqref{one-sol-gen-th}.
	From Theorem 1 in \cite{RS19} it follows that $q^{\mathrm{reg}}=\osmall(1)$ as $t\to\infty$ along any ray $x/t=const$ (notice that condition 2 of the  Theorem implies that
	$\int_{-\infty}^{\xi}
	(1+r_1^{\mathrm{reg}}r_2^{\mathrm{reg}})
	\in\left(-\frac{\pi}{2},\frac{\pi}{2}\right)$).
	Moreover, from the analysis in \cite{RS19} it is clear that
	$(M^{\mathrm{reg}}(x,t,k_0)-I)_{ij}=\osmall(1)$, $i,j=1,2$ as $t\to\infty$, $x/t=const$ for all fixed $k_0\in\mathbb{C}\setminus\mathbb{R}$
	which together with Proposition \ref{M-fixed-k}, Item 2 for $M^{\mathrm{reg}}$ with $n=0$ and \eqref{q-qsol-qreg} implies \eqref{one-sol-gen-th}.
	Finally,  since $q(x,t)$ is the solution of the inverse problem, we have $q(x,0)=q_0(x)$ and thus $q(x,t)$ determined by \eqref{q(x,t)} (which is equivalent to \eqref{q-qsol-qreg}) is a blowing-up Schwartz solution.	
\end{proof}
Applying the Blaschke-Potapov factors recursively, we can generalize the results of Theorem \ref{th-class-one-sol} to the multisoliton case.
\begin{theorem}\label{th-class-L-sol}
	(Blowing-up Schwartz solution: $L$ solitons)
	Consider the Cauchy problem \eqref{iv-nnls} with initial data $q_0(x)\in\mathcal{S}(\mathbb{R})$ satisfying the following properties
	(recall that $\{(\cdot)_s\}_{s=1}^0=\emptyset$):
	\begin{enumerate}
		\item $a_1(k)$ has $L=2N_1+M_1$ simple zeros $\{\I\rho_{1,s}\}_{s=1}^{M_1}\cup\{\zeta_{1,s}, -\bar{\zeta}_{1,s}\}_{s=1}^{N_1}$,
		$M_1,N_1\in\mathbb{N}\cup\{0\}$
		and $a_2(k)$ has $L=2N_2+M_2$ simple zeros
		$\{\I\rho_{2,s}\}_{s=1}^{M_2}\cup\{\zeta_{2,s}, -\bar{\zeta}_{2,s}\}_{s=1}^{N_2}$,
		$M_2,N_2\in\mathbb{N}\cup\{0\}$.
		\item $\|r_j^{\mathrm{reg}}(k)\|_{L^\infty(\mathbb{R})}<1$, $j=1,2$, where
		$$
		r_1^{\mathrm{reg}}(k)=
		\prod\limits_{s=1}^{M_1}
		\frac{k-\I\rho_{1,s}}{k-\I\rho_{2,s}}
		\prod\limits_{s=1}^{N_1}
		\frac{(k-\zeta_{1,s})(k+\bar\zeta_{1,s})}
		{(k-\zeta_{2,s})(k+\bar\zeta_{2,s})}
		r_1(k),
		$$
		and
		$$
		r_2^{\mathrm{reg}}(k)=
		\prod\limits_{s=1}^{M_2}
		\frac{k-\I\rho_{2,s}}{k-\I\rho_{1,s}}
		\prod\limits_{s=1}^{N_2}
		\frac{(k-\zeta_{2,s})(k+\bar\zeta_{2,s})}
		{(k-\zeta_{1,s})(k+\bar\zeta_{1,s})}
		r_2(k).
		$$
	\end{enumerate}
	Let  the 2-vectors
	$g^{(m)}=\left(
	\begin{smallmatrix}
	g_1^{(m)}\\
	g_2^{(m)}
	\end{smallmatrix}
	\right)$
	and
	$h^{(m)}=\left(
	\begin{smallmatrix}
	h_1^{(m)}\\
	h_2^{(m)}
	\end{smallmatrix}
	\right)$, $m=1,\dots,L$ be defined recursively by \eqref{g-j-h-j-L} and let
	\begin{equation}\label{B-L-sol}
	\begin{split}
	\mathcal{B}=\{(x,t)\in\mathbb{R}^2:&\,\,
	g_1^{(m)}(x,t)h_2^{(m)}(x,t)-g_2^{(m)}(x,t)h_1^{(m)}(x,t)=0\\
	&\text{ for some }m \text{ with } 1\le m\le L\}.
	\end{split}
	\end{equation}
	Then
	\begin{enumerate}[(i)]
		\item
		$\mathcal{B}$ is a closed set, which in the case $N_1=N_2=0$ lies in the band $\{(x,t)\in\mathbb{R}^2:|x|<R\}$ for some $R>0$, while if $N_1>0$ or $N_2>0$ it lies in the sector $\{(x,t)\in\mathbb{R}^2:|t|>\alpha|x|-C\}$ for some $0<\alpha<\frac{1}{4\max\limits_{j,m}|\Re \zeta_{j,m}|}$ and $C>0$.
		\item
		$q(x,t)$ defined by \eqref{q(x,t)} via the solution of the basic RH problem is a blowing-up Schwartz solution, with blow-up points in $\mathcal{B}$.
	\end{enumerate}
	Moreover, $q(x,t)$ converges to $L$ soliton (cf.\,\,\cite{Y19}) as $x,t$ become large:
	\begin{equation}\label{L-sol-gen-th}
	q(x,t)=
	2\I\frac{\det M_1(x,t)}{\det M(x,t)}+\osmall(1),
	\end{equation}
	as $|x|^2+|t|^2\to\infty$ avoiding all $(x,t)$ such that $\frac{\det M_1(x,t)}{\det M(x,t)}=\infty$.
	Here $M(x,t)$ and $M_1(x,t)$ are given by \eqref{M-xi-eta} and \eqref{M-1-xi-eta} respectively.
\end{theorem}
\begin{proof}
	To simplify notations, let us denote all zeros and norming constants (see \eqref{z-a} and \eqref{n-c}) by
	\begin{align}
	&\{k_{j,s}\}_{s=1}^{L}=\{\I\rho_{j,s}\}_{s=1}^{M_j}\cup\{\zeta_{j,s}, -\bar{\zeta}_{j,s}\}_{s=1}^{N_j},
	&&j=1,2,\\
	&\{\nu_{j,s}\}_{s=1}^{L}=\{\gamma_{j,s}\}_{s=1}^{M_j}\cup\{\eta_{j,s}, \hat\eta_{j,s}\}_{s=1}^{N_j},
	&&j=1,2.
	\end{align}
	Define $M^{\mathrm{reg}}$ in terms of the solution of the initial RH problem as follows \cite{FT} (cf.\,\,\eqref{M-reg}):
	\begin{equation}\label{M-reg-L}
	M^{\mathrm{reg}}(x,t,k)=B_1\dots B_{L}M(x,t,k)
	\begin{pmatrix}
	1& 0\\
	0& \frac{k-k_{2,L}}{k-k_{1,L}}
	\end{pmatrix}
	\dots
	\begin{pmatrix}
	1& 0\\
	0& \frac{k-k_{2,1}}{k-k_{1,1}}
	\end{pmatrix},
	\end{equation}
	where the $2\times2$ Blaschke-Potapov factors $B_m=B_m(x,t,k)$, $m=1,\dots,L$ have the form
	$$
	B_m(x,t,k)=I+\frac{k_{2,m}-k_{1,m}}{k-k_{2,m}}
	P^{(m)}(x,t),\quad m=1,\dots,L,
	$$
	with projectors $P^{(m)}(x,t)$ ($(P^{(m)})^2(x,t)=P^{(m)}(x,t)$) to be specified.
	Define $P^{(m)}(x,t)$ by
	their kernel and image (cf.\,\,\eqref{ker-im})
	\begin{equation}\label{ker-im-L}
	\ker P^{(m)}(x,t)=
	\mathrm{lin}_\mathbb{C}\{g^{(m)}(x,t)\}, \quad
	\Im P^{(m)}(x,t)=
	\mathrm{lin}_\mathbb{C}\{h^{(m)}(x,t)\},
	\end{equation}
	where the 2-vectors $g^{(m)}$
	and $h^{(m)}$ are obtained recursively by
	(we adopt the notation $\prod\limits_{s=n_1}^{n_2}(\cdot)_s=1$ if $n_2<n_1$):
	\begin{subequations}\label{g-j-h-j-L}
		\begin{align}
		\nonumber
		g^{(m)}=&\,(k_{1,m}-k_{2,m})
		(M^{\mathrm{reg}}_{m-1})^{[1]}(x,t,k_{1,m})\\
		&-\frac{\nu_{1,m}
			\prod\limits_{s=m+1}^{L}
			\frac{k_{1,m}-k_{1,s}}{k_{1,m}-k_{2,s}}}
		{\dot{a}_1(k_{1,m})}
		e^{2\I k_{1,m}x+4\I k_{1,m}^2t}
		(M^{\mathrm{reg}}_{m-1})^{[2]}(x,t,k_{1,m}),\\
		\nonumber
		h^{(m)}=&\,(k_{1,m}-k_{2,m})
		(M^{\mathrm{reg}}_{m-1})^{[2]}(x,t,k_{2,m})\\
		&+\frac{\nu_{2,m}
			\prod\limits_{s=m+1}^{L}
			\frac{k_{2,m}-k_{2,s}}{k_{2,m}-k_{1,s}}}
		{\dot{a}_2(k_{2,m})}
		e^{-2\I k_{2,m}x-4\I k_{2,m}^2t}
		(M^{\mathrm{reg}}_{m-1})^{[1]}(x,t,k_{2,m}),
		\end{align}
	\end{subequations}
	with the following $M^{\mathrm{reg}}_{m-1}(x,t,k)$
	($M^{\mathrm{reg}}_0:=M^{\mathrm{reg}}$)
	\begin{equation}
	\begin{split}
	M_{m-1}^{\mathrm{reg}}(x,t,k)=&\,
	B_{m-1}^{-1}(x,t,k)\dots B_{1}^{-1}\\
	&\times M^{\mathrm{reg}}(x,t,k)
	\begin{pmatrix}
	1&0\\
	0 &\frac{k-k_{1,1}}{k-k_{2,1}}
	\end{pmatrix}
	\dots
	\begin{pmatrix}
	1&0\\
	0 &\frac{k-k_{1,m-1}}{k-k_{2,m-1}}
	\end{pmatrix}.
	\end{split}
	\end{equation}
	Then direct calculations show that
	$M^{\mathrm{reg}}(x,t,k)$ solves a regular RH problem \eqref{M-reg-RH} with $r_j^{\mathrm{reg}}$ defined in condition 2 of the Theorem.
	In view of the assumption $\|r_j^{\mathrm{reg}}\|_{L^\infty}\break <1$,
	this RH problem has a solution.
	Notice that the entries of $P^{(m)}$, $m=1,\dots, L$ can be found by (cf.\,\,\eqref{P-i-j}):
	\begin{equation}\label{P-i-j-L}
	\begin{split}
	&P_{12}^{(m)}=\frac{g_1^{(m)}h_1^{(m)}}
	{g_1^{(m)}h_2^{(m)}-g_2^{(m)}h_1^{(m)}},\,\,
	P_{21}^{(m)}=\frac{-g_2^{(m)}h_2^{(m)}}
	{g_1^{(m)}h_2^{(m)}-g_2^{(m)}h_1^{(m)}},\\
	&P_{11}^{(m)}=-\frac{P_{12}^{(m)}g_2^{(m)}}
	{g_1^{(m)}},\,\,
	P_{22}^{(m)}=-\frac{P_{21}^{(m)}g_1^{(m)}}
	{g_2^{(m)}}.
	\end{split}
	\end{equation}
	
	From \eqref{q(x,t)} and \eqref{M-reg-L} we conclude that  $q(x,t)$ can be determined in terms of $M^{\mathrm{reg}}$ and $P^{(m)}$ as follows:
	\begin{equation}\label{q-qsol-qreg-L}
	q(x,t)=\tilde{q}^{\mathrm{sol}}_{L}(x,t)
	+q^{\mathrm{reg}}(x,t),
	\end{equation}
	where
	\begin{subequations}\label{q-L-sol-qreg}
		\begin{align}
		\label{q-L-sol}
		&\tilde{q}^{\,\mathrm{sol}}_{L}(x,t)=
		\sum\limits_{m=1}^{L}
		2\I(k_{1,m}-k_{2,m})P_{12}^{(m)}(x,t),\\
		&q^{\mathrm{reg}}(x,t)=
		2\I\lim\limits_{k\to\infty}kM_{12}^{\mathrm{reg}}(x,t,k).
		\end{align}
	\end{subequations}
	Now the fact that $q(x,t)$ is a blowing-up Schwartz solution with blow-up points in $\mathcal{B}$ follows from the same arguments as in Theorem \ref{th-class-one-sol}.
	
	Let us show the asymptotics \eqref{L-sol-gen-th}.
	For large $x$ and $t$ we have $M^{\mathrm{reg}}=I+\osmall(1)$, which implies that
	\begin{equation}\label{q-as-L}
	q(x,t)=2\I\lim\limits_{k\to\infty}
	M_{12}^{\mathrm{as}}(x,t,k)+\osmall(1), \quad
	x,t\to\infty,
	\end{equation}
	where $M^{\mathrm{as}}$ satisfies the equality
	(cf.\,\,\eqref{M-reg-L})
	\begin{equation}\label{M-as-L}
	I=B_1\dots B_{L}M^{\mathrm{as}}(x,t,k)
	\begin{pmatrix}
	1& 0\\
	0& \frac{k-k_{2,L}}{k-k_{1,L}}
	\end{pmatrix}
	\dots
	\begin{pmatrix}
	1& 0\\
	0& \frac{k-k_{2,1}}{k-k_{1,1}}
	\end{pmatrix}.
	\end{equation}
	
	Instead of using the representation \eqref{q-L-sol}, which involves recursively determined terms, introduce the partial fraction expansion of
	$B_{L}^{-1}\cdots B_1^{-1}$ (see \cite{FT}):
	\begin{equation}\label{B-A-s}
	B_{L}^{-1}(x,t,k)\cdots B_1^{-1}(x,t,k)
	=I+\sum\limits_{m=1}^{L}
	\frac{A_m(x,t)}{k-k_{1,m}}.
	\end{equation}
	Then from \eqref{q-as-L} and \eqref{B-A-s} we have
	\begin{equation}\label{q-A-s}
	q(x,t)=2\I\sum\limits_{m=1}^{L}(A_{m})_{12}(x,t)
	+\osmall(1),\quad x,t,\to\infty.
	\end{equation}
	
	In order to determine $A_m$, we first prove that it can be represented as follows:
	\begin{equation}\label{A-s}
	A_m(x,t)=z_m(x,t)\xi_m^{T}(x,t),\quad m=1,\dots, L,
	\end{equation}
	with some $2$-vector $z_m(x,t)$, where  the  raw vector $\xi_m^T(x,t)$
	is given by
	\begin{equation}\label{xi-s}
	\begin{split}
	\xi_m^T(x,t)&=\left(
	\frac{\nu_{1,m}
		\prod\limits_{s\neq m,\,s=1}^{L}
		\frac{k_{1,m}-k_{1,s}}{k_{1,m}-k_{2,s}}}
	{(k_{1,m}-k_{2,m})\dot{a}_1(k_{1,m})}
	e^{2\I k_{1,m}x+4\I k^2_{1,m}t},1\right)\\
	&\equiv (\xi_{m,1}(x,t),1),
	\quad m=1,\dots,L.
	\end{split}
	\end{equation}
	Indeed, taking into account that $B_1\dots B_{L}=C_m(x,t)+\ord(k-k_{1,m})$ as $k\to k_{1,m}$, we conclude that $A_mC_m=C_mA_m=0$.
	Notice that since $\ker C_m(x,t)=\mathrm{lin}_{\mathbb{C}}
	\{\left(M^{\mathrm{as}}\right)^{[2]}(x,t,k_{1,m})\}$, it follows that $\mathrm{rang}(C_m)=1$.
	Moreover, from
	\begin{subequations}
		\begin{align}
		&\frac{A_m^{[1]}(x,t)}{k-k_{1,m}}\sim
		\frac{\nu_{1,m}}{\dot{a}_1(k_{1,m})}
		e^{2\I k_{1,m}x+4\I k^2_{1,m}t}
		\frac{\left(M^{\mathrm{as}}\right)^{[2]}(x,t,k_{1,m})}
		{k-k_{1,m}},&& k\to k_{1,m},\\
		&\frac{A_m^{[2]}(x,t)}{k-k_{1,m}}\sim
		(k_{1,m}-k_{2,m})
		\left(\prod\limits_{s\neq m,s=1}^{L}
		\frac{k_{1,m}-k_{2,s}}{k_{1,m}-k_{1,s}}\right)
		\frac{\left(M^{\mathrm{as}}\right)^{[2]}(x,t,k_{1,m})}
		{k-k_{1,m}},&& k\to k_{1,m},
		\end{align}
	\end{subequations}
	one concludes that
	$A_m^{[1]}(x,t)=A_m^{[2]}(x,t)\xi_{m,1}(x,t)$,
	where $\xi_{m,1}(x,t)$ is defined by  \eqref{xi-s}.
	This and the fact that $\mathrm{rang}(C_m)=1$ yield the  representation \eqref{A-s}.
	
	Now we introduce the $2$-vector $\eta_m(x,t)$ by
	\begin{equation}\label{nu-m}
	\begin{split}
	\eta_m(x,t)&=
	\begin{pmatrix}
	\frac{-\nu_{2,m}
		\left(\prod\limits_{s\neq m,s=1}^{L}
		\frac{k_{2,m}-k_{2,s}}{k_{2,m}-k_{1,s}}
		\right)}
	{(k_{2,m}-k_{1,m})\dot{a}_2(k_{2,m})}
	e^{-2\I k_{2,m}x-4\I k^2_{2,m}t}\\
	1
	\end{pmatrix}\\
	&\equiv
	\begin{pmatrix}
	\eta_{m,1}(x,t)\\
	1
	\end{pmatrix},
	\quad m=1,\dots,L.
	\end{split}
	\end{equation}
	Substituting \eqref{B-A-s} in \eqref{M-as-L} with $k=k_{2,m}$ and multiplying by $\eta_m(x,t)$ on the right, we arrive at the following system of linear algebraic equations
	\begin{equation}\label{z-syst}
	\eta_m(x,t)+\sum\limits_{s=1}^{L}
	\frac{\xi_s^T(x,t)\eta_m(x,t)}{k_{2,m}-k_{1,s}}
	z_m(x,t)=0,\quad m=1,\dots,L.
	\end{equation}
	Notice that from \eqref{A-s} it follows that $\left(A_{m}\right)_{12}=z_{m,1}$, where
	$z_m=\left(\begin{smallmatrix}
	z_{m,1}\\
	z_{m,2}
	\end{smallmatrix}\right)$.
	Therefore, using Cramer's rule, from \eqref{q-A-s} and \eqref{z-syst} we have that
	\begin{equation}
	q(x,t)=2\I\frac{\det M_1(x,t)}{\det M(x,t)}+\osmall(1),\quad x,t\to\infty,
	\end{equation}
	where the $(ms)$ element of $M$ is as follows:
	\begin{equation}\label{M-xi-eta}
	M_{ms}(x,t)=\frac{\xi_s^T(x,t)\eta_m(x,t)}{k_{2,m}-k_{1,s}},\quad m,s=1,\dots,L,
	\end{equation}
	and
	\begin{equation}\label{M-1-xi-eta}
	M_1(x,t)=
	\left(
	\begin{array}{@{}c|c@{}}
	M(x,t) &
	\begin{matrix}
	-\eta_{1,1}(x,t)\\
	\dots\\
	-\eta_{L,1}(x,t)
	\end{matrix}\\
	\hline
	\begin{matrix}
	1&\dots&1
	\end{matrix}
	&0
	\end{array}
	\right),
	\end{equation}
	where $\eta_{m,1}$ is given by \eqref{nu-m}.
	
	Finally, since the functions $g_1^{(m)}(x,t)h_2^{(m)}(x,t)-g_2^{(m)}(x,t)h_1^{(m)}(x,t)$, $m=1,\dots,L$ are continuous, the set $\mathcal{B}$ is closed.
	Moreover, taking into account that $\frac{\det{M_1(x,t)}}{\det M(x,t)}\to 0$ as $x,t\to\infty$ along any ray $\xi=\frac{x}{4t}=const$ such that $|\xi|>\xi_0$, where $\xi_0=\max\limits_{j,m}|\Re k_{j,m}|$, one concludes that (i) if $\xi_0=0$ (which is equivalent to $N=0$), then $q(x,t)$ does not blow-up outside the band $\{(x,t)\in\mathbb{R}^2:|x|<R\}$, while (ii) for $\xi_0>0$, $q(x,t)$ does not blow-up outside the sector
	$\{(x,t)\in\mathbb{R}^2:|t|>\alpha|x|-C\}$, where $0<\alpha<\frac{1}{4\xi_0}$ and $C>0$.
\end{proof}
\begin{remark}
	(Concluding remarks to Theorem \ref{th-class-L-sol})\label{concl-th-1}
	\begin{enumerate}[(i)]
		\item
		The blowing-up Schwartz solution derived in Theorem is a conservative solution.
		More precisely, for any $t\in\mathbb{R}$ such that $(\mathbb{R}\times\{t\})\cap\mathcal{B}=\emptyset$, we have that all quantities
		$I_n(t)=\int_{-\infty}^{+\infty}p_n(x,t)\,dx$, $n\in\mathbb{N}$, see \eqref{p-n}, are constants: $I_n(t)=I_n(0)$ for all $n\in\mathbb{N}$ (particularly, $\tilde{M}$ and $\tilde{E}$ defined by \eqref{cons-mass} and \eqref{cons-energy} respectively are conserved).
		Indeed,   $q(\cdot,t)\in\mathcal{S}(\mathbb{R})$
		for such $t$, and since it is obtained by the inverse transform \eqref{q(x,t)}, $q(\cdot,t)$ gives rise to the same spectral functions $a_j(k)$ as the initial data $q_0(x)$.
		Therefore, the values of $I_n(t)$ defined by the expansion of $\ln a_1(k)$ for large $k$, see \eqref{In-a1}, are the same as $I_n(0)$.
		\item
		Since $\mathcal{B}$ is closed, there exists $\varepsilon=\varepsilon(q_0)>0$ such that
		$\mathcal{B}\cap\{(x,t)\in\mathbb{R}^2:|t|<\varepsilon\}=\emptyset$ (see also Remark \ref{char-B} below about the description of the set $\mathcal{B}$).
		Moreover, the blowing-up Schwartz solutions constructed in Theorem \ref{th-class-L-sol} coincide with that obtained in Theorem \ref{th-class-glob} on the maximal interval $(-T_{1,\max}, T_{2,\max})$, $T_{j,\max}>0$
		(by a maximal interval, we mean that $\left(\mathbb{R}\times(-T_{1,\max}, T_{2,\max})\right)\cap\mathcal{B}=\emptyset$
		and there exist $x_j\in\mathbb{R}$ such as $(x_j, T_{j,\max})\in\mathcal{B}$, $j=1,2$).
		\item
		The exact $L$ soliton solutions correspond to the reflectionless case (i.e., $r_j(k)\equiv0$, $j=1,2$).
		So, the class of initial data described in spectral terms in Theorem \ref{th-class-L-sol} contains, in particular, small $L^1$ perturbations of smooth soliton initial profiles (see Proposition \ref{prop-suff-L-sol} below for details).
		\item
		The terms $q^{\mathrm{reg}}$ and $\tilde{q}_{L}^{\mathrm{sol}}$ in \eqref{q-L-sol-qreg} express, respectively, the
		smooth dispersion part and the soliton part of the solution, where the latter bares all possible blow-up points of the solution.
	\end{enumerate}
\end{remark}
We close this Section by giving a sufficient condition 
in the spirit of Proposition \ref{prop-suff} above,
which guarantees that the conditions 1 and 2 in Theorem \ref{th-class-L-sol} are fulfilled.
\begin{proposition}\label{prop-suff-L-sol}
	(Sufficient condition for Items 1,2 in Theorem \ref{th-class-L-sol}).
	Let
	\begin{subequations}
		\begin{align}
		&\alpha_1(k):=
		\prod\limits_{s=1}^{M_1}
		\frac{k-\I\rho_{1,s}}{k-\I\rho_{2,s}}
		\prod\limits_{s=1}^{N_1}
		\frac{(k-\zeta_{1,s})(k+\bar\zeta_{1,s})}
		{(k-\zeta_{2,s})(k+\bar\zeta_{2,s})},\\
		&\alpha_2(k):=
		\prod\limits_{s=1}^{M_2}
		\frac{k-\I\rho_{2,s}}{k-\I\rho_{1,s}}
		\prod\limits_{s=1}^{N_2}
		\frac{(k-\zeta_{2,s})(k+\bar\zeta_{2,s})}
		{(k-\zeta_{1,s})(k+\bar\zeta_{1,s})},
		\end{align}
	\end{subequations}
	and let  $d_j$, $j=1,2$
	be the positive constants
	defined
	by
	\begin{equation}\label{d-j}
	d_j=\inf\limits
	_{k\in\mathbb{R}\cup\mathcal{D}_j}
	\{|\alpha_j(k)|\},\quad j=1,2,
	\end{equation}
	where $\mathcal{D}_1=
	\{k\in\mathbb{C}^+:\alpha_1^\prime(k)=0\}$
	and
	$\mathcal{D}_2=
	\{k\in\mathbb{C}^-:\alpha_2^\prime(k)=0\}$.
	Denote by $q_{L}^{\mathrm{sol}}(x,t)$  the exact soliton solution characterized by the discrete spectrum $\mathcal{Z}=\bigcup\limits_{j=1}^{2}
	\{\I\rho_{j,s}\}_{s=1}^{M_j}
	\cup\{\zeta_{j,s}, -\bar{\zeta}_{j,s}\}_{s=1}^{N_j}$ and the norming constants\\
	$\mathcal{N}(t)=\bigcup\limits_{j=1}^{2}
	\{\gamma_{j,s}e^{(-1)^j4\I\rho_{j,s}^2t}\}
	_{s=1}^{M_j}
	\cup\{\eta_{j,s}e^{(-1)^{j+1}4\I\zeta_{j,s}^2t}\}_{s=1}^{N_j}$:
	\begin{equation}
	q_{L}^{\mathrm{sol}}(x,t)=
	2\I\frac{\det M_1(x,t)}{\det M(x,t)},
	\end{equation}
	where $M(x,t)$ and $M_1(x,t)$ are given by \eqref{M-xi-eta} and \eqref{M-1-xi-eta} respectively with $a_j(k)=\alpha_j(k)$.

	Assume that  $q_0(x)\in\mathcal{S}(\mathbb{R})$ satisfies the following inequality
	\begin{equation}\label{suff-cond-L-sol}
	\|q_0-q_{0,L}^{\mathrm{sol}}\|
	_{L^1(\mathbb{R})}
	\left\{
	C_1^{\infty}
	\left(\|q_0\|_{L^1(\mathbb{R})}
	+\|q_{0,L}^{\mathrm{sol}}\|
	_{L^1(\mathbb{R})}\right)
	+C_2^{\infty}
	I_0\left(2\|q_{0,L}^{\mathrm{sol}}\|
	_{L^1(\mathbb{R})}\right)
	\right\}
	<d,
	\end{equation}
	where $q_{0,L}^{\mathrm{sol}}(x)=
	q_{L}^{\mathrm{sol}}(x,0)$, $I_0(\cdot)$ is a modified Bessel function, $d=\min\{d_1,d_2\}$ and
	$C_j^{\infty}$ are as follows
	\begin{align}
	\label{C-1-infty}
	&C_1^{\infty}=
	\sum\limits_{n=1}^{\infty}
	\sum\limits_{m=0}^{n-1}
	\left(
	\frac{\|q_0\|_{L^1(\mathbb{R})}^{2m}}
	{(m!)^2}\cdot
	\frac{\|q_{0,L}^{\mathrm{sol}}\|
		_{L^1(\mathbb{R})}^{2(n-m-1)}}
	{((n-m-1)!)^2}
	\right)
	,\\
	\nonumber
	&C_2^{\infty}=1
	+\sum\limits_{n=1}^{\infty}
	\left\{
	\sum\limits_{m=0}^{n-1}
	\left(
	\frac{\|q_0\|_{L^1(\mathbb{R})}^{2m}}
	{(m!)^2}\cdot
	\frac{\|q_{0,L}^{\mathrm{sol}}\|
		_{L^1(\mathbb{R})}^{2(n-m)}}
	{((n-m)!)^2}
	+\frac{\|q_0\|_{L^1(\mathbb{R})}^{2m+1}}
	{(m!)^2}\cdot
	\frac{\|q_{0,L}^{\mathrm{sol}}\|
		_{L^1(\mathbb{R})}^{2(n-m)-1}}
	{((n-m-1)!)^2}
	\right)\right.\\
	\label{C-2-infty}
	&
	\left.\qquad\qquad\qquad\,\,\,\,
	+\frac{\|q_0\|_{L^1(\mathbb{R})}^{2n}}{(n!)^2}
	\right\}.
	\end{align}
	Then conditions 1 and 2 in Theorem \ref{th-class-L-sol} hold true.
\end{proposition}
\begin{proof}
	We prove the Proposition for $a_1(k)$ and $r_1(k)$, the proof for $a_2(k)$ and $r_2(k)$ is similar.
	Recall that $a_1(k)$ and $\alpha_1(k)$ can be found as follows (see \eqref{ab-lim})
	\begin{equation}\label{a-1-alpha-1}
	a_{1}(k)=\lim\limits_{x\rightarrow+\infty}
	\psi_1(x,0,k),\quad
	\alpha_{1}(k)=\lim\limits_{x\rightarrow+\infty}
	\psi_1^{\mathrm{sol}}(x,0,k),
	\end{equation}
	where $\psi_1(x,0,k)$ and $\psi_1^{\mathrm{sol}}(x,0,k)$ satisfy the following Volterra integral equations (cf.\,\,\eqref{psi-1-V})
	\begin{subequations}
		\begin{align}
		&\psi_1(x,0,k)=1-\sigma\int_{-\infty}^{x}
		\overline{q_{0}(-y)}
		\int_y^x e^{2\I k(z-y)}
		q_0(z)\,dz\,\psi_1(y,0,k)\,dy,\\
		&\psi_1^{\mathrm{sol}}(x,0,k)
		=1-\sigma\int_{-\infty}^{x}
		\overline{q_{0,L}^{\mathrm{sol}}(-y)}
		\int_y^x e^{2\I k(z-y)}
		q_{0,L}^{\mathrm{sol}}(z)\,dz\,
		\psi_1^{\mathrm{sol}}(y,0,k)\,dy.
		\end{align}
	\end{subequations}
	Taking into account that (here we adopt notation
	$\prod\limits_{i=i_1}^{i_2}(\cdot)=1$ if $i_1>i_2$)
	\begin{equation}
	\prod\limits_{i=1}^{n}u_i
	-\prod\limits_{i=1}^{n}w_i=
	\sum
	\limits_{j=1}^{n}
	\left(
	\prod\limits_{i=1}^{j-1}u_i\cdot
	(u_j-w_j)\cdot
	\prod\limits_{i=j+1}^{n}w_i
	\right),
	\end{equation}
	for any $u_i,w_i\in\mathbb{C}$, from the Neumann series for $\psi_1$ and $\psi_1^{\mathrm{sol}}$ we obtain the following estimate for $k\in\overline{\mathbb{C}^+}$ and
	$x\in\mathbb{R}$:
	\begin{equation}\label{psi-1-psi-sol-1}
	|\psi_1(x,0,k)-\psi_1^{\mathrm{sol}}(x,0,k)|\leq
	\|q_0-q_{0,L}^{\mathrm{sol}}\|_{L^1}
	(\|q_0\|_{L^1}+\|q_{0,L}^{\mathrm{sol}}\|_{L^1})
	C_1(x),
	\end{equation}
	where
	\begin{equation}
	C_1(x)=
	\sum\limits_{n=1}^{\infty}
	\sum\limits_{m=0}^{n-1}
	\left(
	\frac{Q_0^m(x)R_0^m(x)}
	{(m!)^2}\cdot
	\frac{(Q_0^{\mathrm{sol}}(x))^{n-m-1}
		(R_0^{\mathrm{sol}}(x))^{n-m-1}}
	{((n-m-1)!)^2}
	\right),
	\end{equation}
	with
	\begin{align}
	\nonumber
	&Q_0(x)=\int_{-\infty}^{x}|q_0(y)|\,dy,&& R_0(x)=\int_{-\infty}^{x}|q_{0}(-y)|\,dy,\\
	\nonumber
	&Q_0^{\mathrm{sol}}(x)=\int_{-\infty}^{x}
	|q_{0,L}^{\mathrm{sol}}(y)|\,dy,&& R_0^{\mathrm{sol}}(x)=\int_{-\infty}^{x}
	|q_{0,L}^{\mathrm{sol}}(-y)|\,dy.
	\end{align}
	Notice that
	$C_1^{\infty}=\lim\limits_{x\to+\infty}C_1(x)$, see \eqref{C-1-infty}.
	From \eqref{a-1-alpha-1} we have that
	$a_1(k)=\alpha_1(k)+
	\lim\limits_{x\to+\infty}
	(\psi_1(x,0,k)-\psi_1^{\mathrm{sol}}(x,0,k))$,
	which implies that if
	\begin{equation}\label{zer-a-1}
	\|q_0-q_{0,L}^{\mathrm{sol}}\|_{L^1}
	(\|q_0\|_{L^1}+\|q_{0,L}^{\mathrm{sol}}\|_{L^1})
	C_1^{\infty}<d_1,
	\end{equation}
	where $d_1$ is given by \eqref{d-j}, then the only zeros of $a_1(k)$ are at the points $\{\I\rho_{1,s}\}_{s=1}^{M_1}
	\cup\{\zeta_{1,s}, -\bar{\zeta}_{1,s}\}_{s=1}^{N_1}$.
	Arguing similarly for $a_2(k)$ we conclude that if
	\begin{equation}\label{zer-a-2}
	\|q_0-q_{0,L}^{\mathrm{sol}}\|_{L^1}
	(\|q_0\|_{L^1}+\|q_{0,L}^{\mathrm{sol}}\|_{L^1})
	C_1^{\infty}<d_2,
	\end{equation}
	then the only zeros of $a_2(k)$ are at the points $\{\I\rho_{2,s}\}_{s=1}^{M_2}
	\cup\{\zeta_{2,s}, -\bar{\zeta}_{2,s}\}_{s=1}^{N_2}$.
	
	Since soliton solution $q_{L}^{\mathrm{sol}}(x,t)$ corresponds to reflectionless potential, we have
	\begin{equation}
	b(k)=\lim\limits_{x\rightarrow+\infty}
	e^{-2\I kx}\psi_3(x,0,k),\quad
	0=\lim\limits_{x\rightarrow+\infty}
	e^{-2\I kx}\psi_3^{\mathrm{sol}}(x,0,k),
	\end{equation}
	where $\psi_3(x,0,k)$ and $\psi_3^{\mathrm{sol}}(x,0,k)$ satisfy the following Volterra integral equations (see \eqref{psi-1-3})
	\begin{subequations}
		\begin{align}
		\nonumber
		&\psi_3(x,0,k)=
		-\sigma\int_{-\infty}^{x}e^{2\I k(x-y)}
		\overline{q_{0}(-y)}\,dy\\
		&\qquad\qquad\qquad
		-\sigma\int_{-\infty}^{x}
		q_{0}(y)
		\int_y^x e^{2\I k(x-z)}
		\overline{q_0(-z)}\,dz\,\psi_3(y,0,k)\,dy,\\
		\nonumber
		&\psi_3^{\mathrm{sol}}(x,0,k)
		=
		-\sigma\int_{-\infty}^{x}e^{2\I k(x-y)}
		\overline{q_{0}^{\mathrm{sol}}(-y)}\,dy\\
		&\qquad\qquad\qquad
		-\sigma\int_{-\infty}^{x}
		q_{0,L}^{\mathrm{sol}}(y)
		\int_y^x e^{2\I k(x-z)}
		\overline{q_{0,L}^{\mathrm{sol}}(-z)}\,dz\,
		\psi_3^{\mathrm{sol}}(y,0,k)\,dy.
		\end{align}
	\end{subequations}
	Similarly to the estimate for $|\psi_1-\psi_1^{\mathrm{sol}}|$, we have
	\begin{equation}
	|\psi_3(x,0,k)-\psi_3^{\mathrm{sol}}(x,0,k)|\leq
	\|q_0-q_{0,L}^{\mathrm{sol}}\|_{L^1}
	C_2(x),\quad x,k\in\mathbb{R},
	\end{equation}
	where
	\begin{align}
	\nonumber
	C_2(x)=&1
	+\sum\limits_{n=1}^{\infty}
	\left\{
	\sum\limits_{m=0}^{n-1}
	\left(
	\frac{Q_0^m(x) R_0^m(x)}
	{(m!)^2}\cdot
	\frac{
		(Q_0^{\mathrm{sol}}(x))^{(n-m)}
		(R_0^{\mathrm{sol}}(x))^{(n-m)}
	}
	{((n-m)!)^2}
	\right.\right.\\
	\nonumber
	&
	\left.\left.
	+\|q_0\|_{L^1}
	\|q_{0,L}^{\mathrm{sol}}\|
	_{L^1}
	\frac{Q_0^m(x) R_0^m(x)}
	{(m!)^2}\cdot
	\frac{
		(Q_0^{\mathrm{sol}}(x))^{(n-m-1)}
		(R_0^{\mathrm{sol}}(x))^{(n-m-1)}
	}
	{((n-m-1)!)^2}
	\right)\right.\\
	&\left.+\frac{Q_0^n(x) R_0^n(x)}{(n!)^2}
	\right\},
	\end{align}
	Observe that
	$C_2^{\infty}=\lim\limits_{x\to+\infty}C_2(x)$, see \eqref{C-2-infty}.
	Using that $|\alpha_1(k)|\leq I_0\left(2\|q_{0,L}^{\mathrm{sol}}\|
	_{L^1}\right)$ for all $k\in\overline{\mathbb{C}^+}$
	(cf.\,\,\eqref{N-ser-F-1}) and \eqref{psi-1-psi-sol-1},
	we obtain the following estimate
	\begin{align}\label{r-1-reg-est}
	\nonumber
	|r_1^{\mathrm{reg}}(k)|
	&\leq
	\lim\limits_{x\to+\infty}
	|\psi_3(x,0,k)-\psi_3^{\mathrm{sol}}(x,0,k)|
	\left|\frac{\alpha_1(k)}{a_1(k)}\right|\\
	\nonumber
	&\leq
	\|q_0-q_{0,L}^{\mathrm{sol}}\|_{L^1}
	C_2^{\infty}
	\frac{I_0\left(
		2\|q_{0,L}^{\mathrm{sol}}\|_{L^1}\right)}
	{\alpha_1(k)-|a_1(k)-\alpha_1(k)|}\\
	&\leq
	\|q_0-q_{0,L}^{\mathrm{sol}}\|_{L^1}
	C_2^{\infty}
	\frac{I_0\left(
		2\|q_{0,L}^{\mathrm{sol}}\|_{L^1}\right)}
	{d_1-\|q_0-q_{0,L}^{\mathrm{sol}}\|_{L^1}
		(
		\|q_0\|_{L^1}+\|q_{0,L}^{\mathrm{sol}}\|_{L^1}
		)C_1^{\infty}},
	\end{align}
	where $d_1$ is given by \eqref{d-j}.
	Arguing similarly for $r_2^{\mathrm{reg}}(k)$, we have
	\begin{equation}\label{r-2-reg-est}
	|r_2^{\mathrm{reg}}(k)|\leq
	\|q_0-q_{0,L}^{\mathrm{sol}}\|_{L^1}
	C_2^{\infty}
	\frac{I_0\left(
		2\|q_{0,L}^{\mathrm{sol}}\|_{L^1}\right)}
	{d_2-\|q_0-q_{0,L}^{\mathrm{sol}}\|_{L^1}
		(
		\|q_0\|_{L^1}+\|q_{0,L}^{\mathrm{sol}}\|_{L^1}
		)C_1^{\infty}}.
	\end{equation}
	Estimates \eqref{r-1-reg-est} and \eqref{r-2-reg-est} imply that if \eqref{suff-cond-L-sol} holds, then $|r_j^{\mathrm{reg}}(k)|<1$, $j=1,2$ for all $k\in\mathbb{R}$.
	Moreover, \eqref{suff-cond-L-sol} implies \eqref{zer-a-1} and \eqref{zer-a-2}, and therefore conditions 1 and 2 of Theorem \ref{th-class-L-sol} are fulfilled.
\end{proof}
\begin{remark}
	The inequality \eqref{suff-cond-L-sol}
	can be applied in the case $q_{0,L}^{\mathrm{sol}}(x)\equiv0$ as well.
	In this case,  we have
	$a_1(k)=a_2(k)\equiv1$,  $C_1^{\infty}=C_2^{\infty}=I_0\left(2\|q_0\|_{L^1}\right)$, $d=1$, and thus \eqref{suff-cond-L-sol}
	becomes
	\begin{equation}\label{suff-cond-0}
	\|q_0\|_{L^1}(1+\|q_0\|_{L^1})I_0\left(2\|q_0\|_{L^1}\right)<1.
	\end{equation}
	Comparing \eqref{suff-cond-0} with the inequality \eqref{suff-Bessel} in Proposition \ref{prop-suff}
	we observe that \eqref{suff-Bessel} is better
	(less restrictive) than \eqref{suff-cond-0}. Indeed, taking $\|q_0\|_{L^1}=0.532$,  we have that  \eqref{suff-Bessel} holds whereas
	\eqref{suff-cond-0} does not (the l.h.s.\,\,of \eqref{suff-cond-0} is greater than $1.062$).
\end{remark}
\section{Solutions in $H^{1,1}$ with blow-up points }\label{bl-up-H11}
Now let us consider $q_0(x)\in H^{1,1}(\mathbb{R})$.
It is common to define a weak solution as follows.
\begin{definition}(Weak $H^{1,1}$ solution)
	We say that $q(x,t)$ is a weak $H^{1,1}$ solution of the Cauchy problem \eqref{iv-nnls} with $q_0(x)\in H^{1,1}(\mathbb{R})$, if $t\mapsto q(\cdot,t)$ is a continuous map from $\mathbb{R}$ to $H^{1,1}(\mathbb{R})$ and it satisfies the Duhamel formula
	\begin{equation}\label{Duhamel}
	q(x,t)=e^{\I t\bigtriangleup}q_0(x)+2\I\sigma\int_0^t
	e^{\I(t-\tau)\bigtriangleup}
	q^2(x,\tau)\overline{q(-x,\tau)}\,d\tau,
	\end{equation}
	where $\bigtriangleup=\partial_x^2$.
\end{definition}
\subsection{Local well-posedness in $H^{1,1}$}
In \cite{G17} Genoud proved the local existence and uniqueness of the solution of \eqref{iv-nnls} in $H^1(\mathbb{R})$.
Here, using iteration arguments \cite{S63}, we prove the local well-posedness  for the NNLS equation in the weighted space $H^{1,1}(\mathbb{R})$ (cf.\,\,Proposition 3.8 in \cite{T06}).
\begin{proposition}\label{lwp}
	(Local well-posedness in $H^{1,1}$).
	For all $R>0$ there exists $T=T(R)>0$ such that for any
	$q_0(x)\in H^{1,1}(\mathbb{R})$
	with
	$\|q_0\|_{H^{1,1}(\mathbb{R})}<R$,
	the solution $q(\cdot,t)\in H^{1,1}(\mathbb{R})$
	exists and is unique
	for all $t\in[-T,T]$.
\end{proposition}
\begin{proof} We will use the following abstract iteration argument \cite{S63} (cf.\,\,Proposition 1.38 in \cite{T06}):
	\begin{lemma}\label{Seg}\cite{S63}
		Let $\mathcal{N}$ and $\mathcal{S}$ be two Banach spaces and $D:\mathcal{N}\to\mathcal{S}$ be a linear operator with the bound
		$$
		\|D(f)\|_{\mathcal{S}}\leq C_0\|f\|_{\mathcal{N}}\quad
		\text{for all } f\in\mathcal{N},\quad C_0>0.
		$$
		Assume that $N:\mathcal{S}\to\mathcal{N}$, $N(0)=0$ is a nonlinear map which satisfies the Lipschitz condition
		$$
		\|N(q)-N(p)\|_{\mathcal{N}}\leq
		\frac{1}{2C_0}\|u-v\|_{\mathcal{S}},\quad\text{for all }
		q,p\in B_\varepsilon,
		$$
		where $B_\varepsilon=\{q\in\mathcal{S}:
		\|q\|_\mathcal{S}\leq\varepsilon\}$ for some $\varepsilon>0$.
		Then the nonlinear equation
		\begin{equation}\label{abs-nonl}
		q=q_{\mathrm{lin}}+D[N(q)],
		\end{equation}
		has a unique solution $q\in B_\varepsilon$ for all $q_{\mathrm{lin}}\in B_\frac{\varepsilon}{2}$, which satisfies
		$$
		\|q\|_{\mathcal{S}}\leq 2\|q_{\mathrm{lin}}\|_{\mathcal{S}}.
		$$
	\end{lemma}
	We take $\mathcal{S}=\mathcal{N}=\mathcal{N}_T$, where
	\begin{equation}
	\mathcal{N}_T=\{q(\cdot,t)\in H^{1,1}(\mathbb{R}):
	t\mapsto H^{1,1}(\mathbb{R})\text{ is continuous for } t\in[-T,T]\},
	\end{equation}
	with the norm
	\begin{equation}
	\|f(x,t)\|_{\mathcal{N}_T}=
	\sup\limits_{t\in[-T,T]}\|f(x,t)\|_{H^{1,1}(\mathbb{R})}.
	\end{equation}
	We write the nonlinear equation \eqref{Duhamel} in the form \eqref{abs-nonl} with $q_{\mathrm{lin}}=e^{\I t\bigtriangleup}q_0(x)$ and with the following linear operator $D$ and nonlinear map $N$:
	\begin{equation}
	D(f(x,t))=\int_0^te^{\I(t-\tau)\bigtriangleup}f(x,\tau)\,d\tau,
	\quad
	N(q(x,t))=2\I\sigma q^2(x,t)\overline{q(-x,t)}.
	\end{equation}
	Using the Strichartz estimate (see, e.g., \cite{T06})
	\begin{equation}\label{Str-est}
	\|e^{\I t\bigtriangleup}f\|_{H^{1,1}(\mathbb{R})}
	\leq\hat C(1+|t|^2)^{1/2}\|f\|_{H^{1,1}(\mathbb{R})},\quad f\in H^{1,1}(\mathbb{R}),\quad \hat C>0,
	\end{equation}
	we conclude that $D$ is an operator from $\mathcal{N}_T$ to $\mathcal{N}_T$:
	\begin{equation}\label{D-f-est}
	\begin{split}
	\|D(f)\|_{\mathcal{N}_T}&\leq\sup\limits_t\int_0^{t}
	\|e^{\I(t-\tau)\bigtriangleup} f(x,\tau)\|_{H^{1,1}}\,d\tau\\
	&\leq\hat C\int_0^{T}(1+|T-\tau|^2)^{\frac{1}{2}}\,d\tau
	\|f\|_{\mathcal{N}_T}
	\leq C_1T\|f\|_{\mathcal{N}_T},
	\end{split}
	\end{equation}
	where $C_1>0$ does not depend on $T$.
	Taking into account the algebra property of the space $H^{1,1}$:
	\begin{equation}
	\|fg\|_{H^{1,1}(\mathbb{R})}\leq
	C\|f\|_{H^{1,1}(\mathbb{R})}\|g\|_{H^{1,1}(\mathbb{R})},\quad
	f,g\in H^{1,1}(\mathbb{R}),\quad C>0,
	\end{equation}
	one obtains the following estimate for all fixed $t$ (here we drop the arguments in $f(x,t)$ and $g(x,t)$):
	\begin{align}\label{N-est}
	\nonumber
	\|N(f)-N(g)\|_{H^{1,1}}
	\leq &2(\|f(x)\overline{f(-x)}\|_{H^{1,1}}
	\|f-g\|_{H^{1,1}}\\
	\nonumber
	&\quad+\|f(x)\overline{f(-x)}g(x)-g^2(x)\overline{g(-x)}\|_{H^{1,1}})\\
	\nonumber
	&\leq 2(\|f(x)\overline{f(-x)}\|_{H^{1,1}}+\|fg\|_{H^{1,1}}
	\|\overline{f(-x)}-\overline{g(-x)}\|_{H^{1,1}}\\
	\nonumber
	&\quad+\|g(x)\overline{g(-x)}\|_{H^{1,1}}\|f-g\|_{H^{1,1}})\\
	& \leq C (\|f\|^2_{H^{1,1}}+\|g\|^2_{H^{1,1}})\|f-g\|_{H^{1,1}}.
	\end{align}
	The latter implies that
	\begin{equation}
	\|N(f)-N(g)\|_{\mathcal{N}_T}\leq
	C_2\hat R^2\|f-g\|_{\mathcal{N}_T},\quad
	\|f\|_{\mathcal{N}_T},\|g\|_{\mathcal{N}_T}\leq\hat R,
	\end{equation}
	where $C_2>0$ does not depend on $\hat R$.
	Let a small $T>0$ be such that $C_2\hat R^2\leq\frac{1}{2C_1T}$ (see \eqref{D-f-est}).
	Then, by Lemma \ref{Seg}, for all $e^{\I t\bigtriangleup}q_0(x)$ such that $\|e^{\I t\bigtriangleup}q_0\|_{\mathcal{N}_T}
	\leq\frac{\hat R}{2}$, there exists a unique $q(x,t)$ that satisfies \eqref{Duhamel}.
	Therefore, it remains to determine $\hat R$ in terms of $R$.
	Assuming that $T<1$ and applying \eqref{Str-est} we arrive at
	\begin{equation}
	\|e^{\I t\bigtriangleup} q_0\|_{\mathcal{N}_T}\leq\sqrt{2}\hat C\|q_0\|_{H^{1,1}}\leq\sqrt{2}\hat C R,
	\end{equation}
	which implies that we can take $\hat R=2\sqrt{2}\hat C R$.
\end{proof}
\subsection{Global solutions in $H^{1,1}$}
\label{bl-up-H11-glob}
For presenting global weak solutions of \eqref{iv-nnls},
we will use the following result \cite{Z98} about the direct $\mathcal{R}$ and inverse $\mathcal{I}$ maps (see Section \ref{sec2.2}) in $H^{1,1}(\mathbb{R})$.
\begin{lemma}\label{dir-inv}\cite{Z98}
	Let $\mathcal{D}\subset H^{1,1}(\mathbb{R})$ be such that for all $q_0(x)\in\mathcal{D}$,  the corresponding basic RH problem (i)  does not have residue conditions  and (ii) has a solution.
	Then  the direct and inverse transforms
	are  Lipschitz continuous:
	for all $q_0,\check q_0\in\mathcal{D}$,
	\begin{subequations}
		\begin{align}
		\label{dir-est}
		&\|r_j(k)-\check r_j(k)\|_{H^{1,1}(\mathbb{R})}
		\leq C\|q_0(x)-\check q_0(x)\|_{H^{1,1}(\mathbb{R})},&& j=1,2,\,\, C>0,\\
		\label{inv-est}
		&\|q(\cdot,t)-\check q(\cdot,t)\|_{H^{1,1}(\mathbb{R})}
		\leq C(t)\|r_j(k)-\check r_j(k)\|_{H^{1,1}(\mathbb{R})},&& j=1,2,\,\, C(t)>0.
		\end{align}
	\end{subequations}
	Here $C(t)<C(1+T)$, 
		$C>0$, $t\in[-T,T]$ for any $T>0$, $r_j$ and $\check r_j$, $j=1,2$   are the reflection coefficients
	associated with
	$q_0$ and $\check q_0$ respectively and $q(x,t)$ and $\check q(x,t)$ are recovered from the solutions of the  RH problems determined by $r_j$ and $\check r_j$.
\end{lemma}
\begin{remark}
	In contrast with \cite{Z98, DZ03}, we consider the inverse map for
	$q$ that depends on the  parameter $t$, and we give the proof of \eqref{inv-est} in Proposition \ref{r-q-Lip}.
	On the other hand, it is enough for our purposes to have the Lipschitz continuity for the direct map for the potential at $t=0$ only, which can be proved for the general AKNS system by following the same lines as in \cite{Z98}, where the conventional NLS equation was considered (see also \cite{DZ03}).
	Finally, we notice that since (i) under small $H^{1,1}$ perturbations of the potential, the discrete spectrum does not arise and (ii) a small perturbation of an invertible operator is also invertible, the set $\mathcal{D}$ is an open subset of $H^{1,1}(\mathbb{R})$.
\end{remark}

As in the case of a classical solution, we first prove global solvability in $H^{1,1}$ for small initial data (cf.\,\,\cite{ZF22}).
In the theorem below we prove that under certain conditions on $q_0$, the recovered, by the inverse transform \eqref{q(x,t)}, potential is a global weak solution of \eqref{iv-nnls}.
In turn, the uniqueness will follow from Proposition \ref{lwp}.
\begin{theorem}\label{H-gl-sol}($H^{1,1}$ global solution: no solitons).
	Suppose that $q_0(x)\in H^{1,1}(\mathbb{R})$
	satisfies properties 1. and 2. in Theorem \ref{th-class-glob}.
	
	Then the Cauchy problem \eqref{iv-nnls} has the unique global weak $H^{1,1}$ solution, which can be obtained by \eqref{q(x,t)}
	in terms of the solution of the basic RH problem.
\end{theorem}
\begin{proof}
	Due to Lemma \ref{dir-inv}, we have   $r_j(k)\in H^{1,1}(\mathbb{R})$, $j=1,2$.
	Moreover, since $e^{(-1)^{j+1}4\I (\cdot)^2t}r_j(\cdot)\in H^{1,1}(\mathbb{R})$ and $\|e^{(-1)^{j+1}4\I (\cdot)^2t}r_j(\cdot)\|_{L^{\infty}(\mathbb{R})}<1$,
	it follows that the basic RH problem is solvable for all fixed $t$ and $q(\cdot,t)$ recovered by \eqref{q(x,t)}  belongs to
	$H^{1,1}(\mathbb{R})$ for all fixed $t\in\mathbb{R}$.
	Let us prove that $q(x,t)$ is a unique global solution of \eqref{iv-nnls}.
	
	Consider a sequence $q_{0n}(x)\in\mathcal{S}(\mathbb{R})$, $n\in\mathbb{N}$, such that $\|q_{0n}-q_0\|_{H^{1,1}}\to0$ as $n\to\infty$.
	Let $r_{j,n}(k)$ denote the reflection coefficients corresponding to $q_{0n}$.
	In view of \eqref{dir-est} in Lemma \ref{dir-inv}, $\|r_j-r_{j,n}\|_{H^{1,1}}\to0$ as $n\to\infty$, and we can choose $q_{0n}$ such that $\|r_{j,n}\|_{L^{\infty}}<1$ for all $n\in\mathbb{N}$.
	Then, due to Theorem \ref{th-class-glob}, there exist  a sequence $q_n(x,t)$ of  unique classical solutions with $q_n(x,0)=q_{0n}$, and \eqref{inv-est} implies that $q_n$ converges to $q$: for all $T\in\mathbb{R}$,
	\begin{align}\label{q_n-r_n}
	\sup\limits_{t\in[-T,T]}
	\|q_{n}(\cdot,t)-q(\cdot,t)\|
	_{H^{1,1}(\mathbb{R})}
	&\leq C
	\|r_{j,n}(\cdot)-r_j(\cdot)\|
	_{H^{1,1}(\mathbb{R})}\to 0,\quad n\to\infty.
	\end{align}
	Since $q_n(x,t)$, being
	classical solutions of the NNLS
	equation,
	satisfy \eqref{Duhamel}, it follows that  the limit $q(x,t)$
	satisfies \eqref{Duhamel} as well and thus is a weak $H^{1,1}$ solution for $t\in[-T,T]$ for all $T>0$.

	Finally, due to the local well-posedness result (see Proposition \ref{lwp}), this solution is unique.
\end{proof}
\begin{remark}\label{suff-cond-small-iv}
	A sufficient condition for
	conditions 1 and 2 in Theorem \ref{H-gl-sol} to hold can be written in terms of the $H^{1,1}$ norm of $q_0$ as follows:
	\begin{equation}\label{suff-Bessel-H}
	\frac{1}{2}(2\|q_0\|_{H^{1,1}(\mathbb{R})}+1)
	I_0(4\|q_0\|_{H^{1,1}(\mathbb{R})})
	<1,
	\end{equation}
	which is satisfied, in particular, if $\|q_0\|_{H^{1,1}(\mathbb{R})}\leq 0.266$ and is violated for $\|q_0\|_{H^{1,1}(\mathbb{R})}\break \geq 0.267$.
	
	Indeed, applying the H\"older inequality,  the following estimate for  $f(x)\in H^{1,1}(\mathbb{R})$ holds true:
	\begin{align}
	\nonumber
	\|f\|_{L^1(\mathbb{R})}&=\int_{-\infty}^{+\infty}
	\frac{1}{1+|x|}(1+|x|)|f(x)|\,dx\\
	\nonumber
	&\leq
	\left(\int_{-\infty}^{+\infty}
	\frac{1}{(1+|x|)^2}\,dx\right)^{\frac{1}{2}}
	\left(\int_{-\infty}^{+\infty}
	(1+|x|)^2|f(x)|^2\,dx\right)^{\frac{1}{2}}\\
	\label{L1-H11}
	&\leq \sqrt{2}\left(\int_{-\infty}^{+\infty}
	(1+2|x|+x^2)|f(x)|^2\,dx\right)^{\frac{1}{2}}\leq
	2\|f\|_{H^{0,1}(\mathbb{R})},
	\end{align}
	which together with \eqref{suff-Bessel} implies \eqref{suff-Bessel-H}.
	Notice also that \eqref{L1-H11} implies that initial data satisfying assumptions of Theorem \ref{H-gl-sol} form an open set in $H^{1,1}$.
\end{remark}
\begin{remark}\label{H1-H-11}
	In \cite{G17} Genoud constructed one soliton, blowing-up solution, which can have arbitrary small $H^1$ norm.
	However, smallness of the initial data in the weighted Sobolev space $H^{1,1}$ implies that the solution does not have a soliton  part (see the  sufficient condition \eqref{suff-Bessel-H}
	for conditions 1 and 2 in Theorem \ref{H-gl-sol}) and the weak $H^{1,1}$ solution
	associated with such initial data exists globally in time.
\end{remark}
\begin{remark}\label{cons-laws}
	It is well-nown that conservation laws constitute
	one of the most important tools for establishing global solvability for nonlinear PDEs.
	Particularly, for the conventional NLS equation, the conservation of total mass and energy
	\begin{equation*}
	\begin{split}
	&M[q]=\int_{-\infty}^{+\infty}|q(x,t)|^2\,dx=const,\\
	&E[q]=\int_{-\infty}^{+\infty}
	(|\partial_xq(x,t)|^2-\sigma|q(x,t)|^4)\,dx=const,
	\end{split}
	\end{equation*}
	allows to obtain \textit{a priory} estimates and to show that the solution exists globally.
	In contrast with this,
	the corresponding conservation laws
	in the nonlocal case have the form \eqref{cons-mass} and \eqref{cons-energy}, 
	which do not preserve any reasonable norm and, in addition, can be negative.
	In this paper, in order to  obtain the existence of a global conservative solution (which can have finite-time blow-ups), we use an alternative way, namely, based on the RH problem approach.
	
\end{remark}
\subsection{Solutions in $H^{1,1}$ beyond blow-up}
In Theorem \ref{H-gl-sol}, we showed that a weak $H^{1,1}$ solution of \eqref{iv-nnls} with small initial data can be determined with the help of IST,
in terms of the solution of the associated RH problem.
However, in the case when the initial data is no longer satisfy the smallness assumption and whose evolution can exhibit solitons, the solution, in general, can blow-up in finite time (see \eqref{one-soliton} and Theorem \ref{th-class-L-sol}).
Therefore, for such initial values, the  solution in the sense of \eqref{Duhamel} is no longer exists globally, and a new concept of solution is needed.
Below we define it through the associated RH problem and call it \textit{pointwise $H^{1,1}$ solution}.
\begin{definition}\label{blow-up-H11}
	(Pointwise $H^{1,1}$ solution)\label{def-H11}
	Consider the Cauchy problem \eqref{iv-nnls} with $q_0(x)\in H^{1,1}(\mathbb{R})$, which satisfies Assumptions 1 and 2.
	Then for every $(x,t)\in\mathbb{R}^2$ such that the basic RH problem has a $L^2$ solution $M(x,t,\cdot)$, we say that $q(x,t)$ defined by \eqref{q(x,t)} is a pointwise $H^{1,1}$ solution of  problem \eqref{iv-nnls}.
\end{definition}
\begin{remark}\label{RH-weak-H11-weak}
	Suppose that the RH problem associated to \eqref{iv-nnls}  has a solution for all parameters
	$x$ and $t$ in a band
	$(x,t)\in\mathbb{R}\times[-T_1,T_2]$, $T_j>0$, $j=1,2$.
	Then the resulting pointwise $H^{1,1}$ solution coincides with the conventional weak $H^{1,1}$ solution \eqref{Duhamel} for $t\in[-T_1,T_2]$.
	This fact follows from the arguments used in the proof of Theorem \ref{H-gl-sol}.
	Notice, however, that if
	a (local) weak $H^{1,1}$ solution \eqref{Duhamel}
	exists, this does not imply automatically  the solvability of the
	associated RH problem and, consequently,
	does not guarantee the existence of a pointwise $H^{1,1}$ solution.
\end{remark}
Now let us show that for a class of initial data (which can be viewed as perturbations of exact soliton solutions), the  solution of
\eqref{iv-nnls}
in the sense of Definition \ref{def-H11} can be expressed
as  a limit in $H^{1,1}$  of the blowing-up Schwartz solutions (see Definition \ref{blow-up-S}).
\begin{theorem}\label{th-H11-one-sol}
	(Pointwise $H^{1,1}$ solution: one soliton).
	Consider the Cauchy problem \eqref{iv-nnls} with $\sigma=1$ and initial data $q_0(x)\in H^{1,1}(\mathbb{R})$ satisfying properties 1. and 2. in Theorem \ref{th-class-one-sol}.
	Let $\mathcal{B}$ be defined by \eqref{B-one-sol}.
	
	Then $\mathcal{B}$ satisfies property (i) in Theorem \ref{th-class-one-sol} and
	\begin{enumerate}[(i)]
		\item
		$q(x,t)$ defined by \eqref{q(x,t)} via the solution of the  basic RH problem,  is a pointwise $H^{1,1}$ solution for all $(x,t)\in\mathbb{R}^2\setminus\mathcal{B}$, while  the solution blows-up at all $(x,t)\in\mathcal{B}$: $q(x,t)=\infty$.
		\item
		Let
		$\mathcal{B}_t=\{x\in\mathbb{R}:
		(x,t)\in\mathcal{B}\}$,
		$\mathcal{B}_t^\varepsilon=\{x\in\mathbb{R}: \mathrm{dist}(x,\mathcal{B}_t)<\varepsilon\}$
		and
		$\hat{\mathcal{B}}_t^\varepsilon=
		\cup_{|t^{\prime}-t|<\varepsilon}
		\mathcal{B}_{t^{\prime}}^\varepsilon$
		(here if $\mathcal{B}_t$ is an empty set, then $\mathcal{B}_t^\varepsilon$ is also empty).
		Then for any $\varepsilon>0$ and
		for any $q_{0n}(x)\in\mathcal{S}(\mathbb{R})$ such that $q_{0n}\to q_0$ in $H^{1,1}(\mathbb{R})$, the corresponding blowing-up Schwartz solutions $q_{n}(x,t)$ approximate, uniformly w.r.t.\,\,$t$, the  pointwise $H^{1,1}$ solution $q(x,t)$
		in the following sense:
		\begin{equation}\label{unif-bl-up}
		\sup\limits_{t\in[-T,T]}
		\|q_n(\cdot,t)-q(\cdot,t)\|
		_{H^{1,1}(\mathbb{R}
		\setminus\hat{\mathcal{B}}_t^\varepsilon)}
		\to 0,\quad
		n\to\infty,
		\end{equation}
		for all $T>0$.
	\end{enumerate}
	
	Moreover, $q(x,t)$ has the asymptotics \eqref{one-sol-gen-th} as $x$ and $t$ become large.
\end{theorem}
\begin{proof}
	Using Blaschke-Potapov factors similarly as in Theorem \ref{th-class-one-sol}, we obtain for $q(x,t)$ the same  representation as \eqref{q-qsol-qreg}:
	\begin{equation}\label{q-qsol-qreg-w}
	q(x,t)=\tilde{q}^{\mathrm{sol}}_1(x,t)
	+q^{\mathrm{reg}}(x,t),
	\end{equation}
	where $\tilde{q}^{\mathrm{sol}}_1$ and $q^{\mathrm{reg}}$ are given by \eqref{q-1-sol-qreg}.
	In view of condition 2 of the Theorem, the RH problem for $M^{\mathrm{reg}}$ is solvable for all $(x,t)\in\mathbb{R}^2$.
	Therefore, from \eqref{M-reg} it follows that $M$ has a solution for all $(x,t)\in\mathbb{R}^2\setminus\mathcal{B}$
	and thus  $q(x,t)$ determined by \eqref{q(x,t)} is a pointwise $H^{1,1}$ solution (see Definition \ref{blow-up-H11}) with blow-up points in $\mathcal{B}$.
	
	As in Theorem \ref{th-class-one-sol}, using Proposition \ref{M-fixed-k}, Item 1 and the asymptotics $q^{\mathrm{reg}}=(M^{\mathrm{reg}}(x,t,k_0)-I)=\osmall(1)$ as $t\to\infty$ along any ray $x/t=const$ (the latter follows from the analysis presented in \cite{LYF22} in the case without discrete spectrum, where the authors use the Dbar approach \cite{MM08}), we obtain that $q(x,t)$ has the asymptotic behavior \eqref{one-sol-gen-th} and thus $\mathcal{B}$ lies in the band $\{(x,t)\in\mathbb{R}^2:|x|<R\}$ for some $R>0$.
	Moreover, from Remark \ref{cont-M} it follows that the function $g_1h_2-g_2h_1$ is continuous and thus $\mathcal{B}$ is closed.
	
	Now we prove \eqref{unif-bl-up}.
	Let $q_n=\tilde{q}^{\mathrm{sol}}_{1,n}
	+q^{\mathrm{reg}}_n$, where
	(cf.\,\,\eqref{q-qsol-qreg} and \eqref{q-qsol-qreg-w})
	$$
	\tilde{q}^{\mathrm{sol}}_{1,n}=-2(\rho_{1,1}-\rho_{2,1})P_{12,n}(x,t),
	$$
	with $P_{12,n}=
	\frac{g_{1,n}h_{1,n}}
	{g_{1,n}h_{2,n}-g_{2,n}h_{1,n}}$.
	From Proposition \ref{r-q-Lip} we have that (cf.\,\,\eqref{q_n-r_n})
	$$
	\sup\limits_{t\in[-T,T]}
	\|q_{n}^{\mathrm{reg}}(\cdot,t)
	-q^{\mathrm{reg}}(\cdot,t)\|_{H^{1,1}(\mathbb{R})}
	\to 0,\quad n\to\infty.
	$$
	Applying the same arguments as in Remark \ref{cont-M} we conclude that
	$(g_{1,n}h_{2,n}-g_{2,n}h_{1,n})$
	does not have zeros outside an $\varepsilon$-neighborhood of the zeros of
	$(g_1h_2-g_2h_1)$ for all $n\geq N$, where $N=N(\varepsilon)>0$ is large enough.
	Therefore, for all $t\in[-T,T]$ the functions
	$(g_{1,n}h_{2,n}-g_{2,n}h_{1,n})^{-1}(\cdot,t)$
	and
	$(g_1h_2-g_2h_1)^{-1}(\cdot,t)$ are uniformly bounded for $(\cdot)\in\mathbb{R}\setminus\hat{\mathcal{B}}_t^{\varepsilon}$. Thus, combining  the fact that $H^{1,1}$ is an algebra with Item 1 of Proposition \ref{M-fixed-k} we have
	$$
	\sup\limits_{t\in[-T,T]}
	\|P_{12,n}(\cdot,t)
	-P_{12}(\cdot,t)
	\|_{H^{1,1}(\mathbb{R}
	\setminus\hat{\mathcal{B}}_t^{\varepsilon})}\to 0,
	\quad n\to\infty,
	$$
	which implies \eqref{unif-bl-up}.
\end{proof}

Similarly to   Schwartz blow-up solutions, we can generalize Theorem \ref{th-H11-one-sol} to the case with any number of solitons (cf.\,\,Theorem \ref{th-class-L-sol}).
\begin{theorem}\label{th-H11-L-sol}
	(Pointwise $H^{1,1}$ solution: $L$ solitons)
	Consider the Cauchy problem \eqref{iv-nnls} with the initial data $q_0(x)\in H^{1,1}(\mathbb{R})$ satisfying properties 1. and 2. in Theorem \ref{th-class-L-sol}.
	Let $\mathcal{B}$ be defined by \eqref{B-L-sol}.
	
	Then $\mathcal{B}$ satisfies property (i) in Theorem \ref{th-class-L-sol} and
	\begin{enumerate}[(i)]
		\item
		$q(x,t)$ defined by \eqref{q(x,t)} via the solution of the  basic RH problem is a pointwise $H^{1,1}$ solution for all $(x,t)\in\mathbb{R}^2\setminus\mathcal{B}$, while  the solution blows-up at all $(x,t)\in\mathcal{B}$: $q(x,t)=\infty$.
		\item
		For any $q_{0n}(x)\in\mathcal{S}(\mathbb{R})$ such that $q_{0n}\to q_0$ in $H^{1,1}(\mathbb{R})$, the corresponding blowing-up Schwartz solutions $q_{n}(x,t)$ approximate $q(x,t)$ uniformly w.r.t.\,\,$t$
		in the sense \eqref{unif-bl-up}.
	\end{enumerate}
	
	Moreover, $q(x,t)$ has the asymptotics \eqref{L-sol-gen-th} as $x$ and $t$ become large.
\end{theorem}
\begin{proof}
	Follows  the same arguments as in Theorems \ref{th-class-L-sol} and \ref{th-H11-one-sol}.
\end{proof}
\begin{remark} (Concluding remarks to Theorem \ref{th-H11-L-sol})
	\label{concl-th-2}
	\begin{enumerate}[(i)]
		\item
		Arguing similarly as in Remark \ref{concl-th-1}, Item (i), we conclude that a pointwise $H^{1,1}$ solution derived in Theorem \ref{th-H11-L-sol} via the basic RH problem is a conservative solution, i.e.,
		for any $t\in\mathbb{R}$ such that $(\mathbb{R}\times\{t\})\cap\mathcal{B}=\emptyset$, the quantities
		$I_1(t)$, $I_2(t)$ and $I_3(t)$ are constants: $I_j(t)=I_j(0)$, $j=1,2,3$.
		In particular,  the ``mass'' and ``energy'', see \eqref{cons-mass} and \eqref{cons-energy}, are conserved:
		$$
		\tilde{M}[q_0(\cdot)]=\tilde{M}[q(\cdot,t)],\quad
		\tilde{E}[q_0(\cdot)]=\tilde{E}[q(\cdot,t)].
		$$
		\item
		Since $\mathcal{B}$ is closed, there exists $\varepsilon=\varepsilon(q_0)>0$ such that
		$\mathcal{B}\cap\{(x,t)\in\mathbb{R}^2:|t|<\varepsilon\}=\emptyset$.
		Moreover, the pointwise $H^{1,1}$ solution in Theorem \ref{th-H11-L-sol} coincides with the weak $H^{1,1}$ solution \eqref{Duhamel} on the maximal interval $(-T_{1,max},T_{2,max})$, $T_{j,max}>0$ of existence of the latter, see Remark \ref{RH-weak-H11-weak}.
		\item
		Recall that the $L$ soliton solutions correspond to zero reflection coefficients $r_j\equiv 0$.
		In view of inequalities \eqref{L1-H11} and
		\eqref{suff-cond-L-sol}, if
		\begin{equation}\label{suff-cond-L-sol-H11}
		2\|q_0-q_{0,L}^{\mathrm{sol}}\|
		_{H^{1,1}}
		\left\{
		2C_1^{\infty}
		\left(\|q_0\|_{H^{1,1}}
		+\|q_{0,L}^{\mathrm{sol}}\|
		_{H^{1,1}}\right)
		+C_2^{\infty}
		I_0\left(4\|q_{0,L}^{\mathrm{sol}}\|
		_{H^{1,1}}\right)
		\right\}
		<d,
		\end{equation}
		where the initial $L$ soliton profile $q_{0,L}^{\mathrm{sol}}(x)$ and positive constants $C_j^{\infty}$, $j=1,2$ and $d$ are given in Proposition \ref{prop-suff-L-sol}, then the spectral functions associated with the initial data $q_0(x)\in H^{1,1}(\mathbb{R})$ satisfy conditions 1 and 2 in Theorem \ref{th-H11-L-sol}.
	\end{enumerate}
\end{remark}
\begin{remark}\label{char-B}
	(Description of  $\mathcal{B}$)
	The set of blow-up points $\mathcal{B}$ defined in Theorems \ref{th-class-L-sol} and \ref{th-H11-L-sol} consists of $(x^\prime,t^\prime)$ which satisfy the system of two equations of the form
	$\left\{\begin{smallmatrix}
	F_1(x,t)=0\\
	F_2(x,t)=0
	\end{smallmatrix}\right.,$
	with real $F_j$, $j=1,2$ (see \eqref{B-L-sol}).
	Therefore, if for all $(x^\prime,t^\prime)\in\mathcal{B}$ the Jacobian
	$\left|\begin{smallmatrix}
	\partial_xF_1(x^\prime,t^\prime)&
	\partial_tF_1(x^\prime,t^\prime)\\
	\partial_xF_2(x^\prime,t^\prime)&
	\partial_tF_2(x^\prime,t^\prime)
	\end{smallmatrix}\right|$ exists and is not equal to zero, then, by the Preimage theorem, $\mathcal{B}$ is a manifold of dimension zero.
	Consequently, in this case, $\mathcal{B}$ is a closed discrete set.
\end{remark}

\section*{Acknowledgments}
We thank the anonymous referees for valuable comments which allowed us to improve the manuscript.
Dmitry Shepelsky acknowledges support from the National Academy of Sciences of Ukraine, Grant Agreement No. 0122U111111.

\appendix
\begin{appendices}
\section{Appendix A}
\label{AppA}
\begin{proposition}\label{inv-x}
	Let $M(x,t,k)$ be the solution of the basic RH problem
	and assume that
	it has a partial derivative w.r.t. $x$ for some $x,t\in\mathbb{R}$.
	Then $M(x,t,k)e^{-\I kx\sigma_3}$ satisfies the $x$-equation of the Lax pair \eqref{Lax} with $q(x,t)$ defined by \eqref{q(x,t)}.
\end{proposition}
\begin{proof}
	Let $M(x,t,k)=I+\frac{Q_1(x,t)}{k}+\ord(k^{-2})$ as $k\to\infty$.
	In these notations $q(x,t)=2\I Q_{1,12}(x,t)$ and from symmetry \eqref{M-sym-as} it follows that $Q_{1,21}(x,t)=-\sigma\overline{Q_{1,12}(-x,t)}$ and $\overline{q(-x,t)}=2\I\sigma Q_{1,21}(x,t)$ (cf.\,\,\eqref{q(-x,t)}).
	
	Introduce $\hat M(x,t,k)=M(x,t,k)e^{-\I kx\sigma_3}$.
	Then $\hat M$ satisfies the RH problem with the jump matrix and residue conditions which do not depend on $x$. Therefore $\partial_x\hat M\cdot\hat M^{-1}$ has neither jump nor singular points for $k\in\mathbb{C}$,
	while as $ k\to\infty$
	\begin{equation}\label{M-x-eq}
	\partial_x\hat M(x,t,k)\cdot\hat M^{-1}(x,t,k)=-\I k\sigma_3-\I[Q_1(x,t),\sigma_3]+\ord(k^{-1}).
	\end{equation}
	By the Liouville theorem we have  $\partial_x\hat M+\I k\sigma_3\hat M=
	\left(
	\begin{smallmatrix}
	0& 2\I Q_{1,12}(x,t)\\
	-2\I Q_{1,21}(x,t)& 0
	\end{smallmatrix}
	\right) \hat M$.
\end{proof}
\begin{corollary}
	Let $M(x,t,k)=I+\frac{Q_1(x,t)}{k}+\frac{Q_2(x,t)}{k^2}+\ord(k^{-3})$ as $k\to\infty$. Then from the $x$-equation we have
	\begin{equation}\label{Q_2}
	\partial_x{Q_1}(x,t)+\I[\sigma_3,Q_2(x,t)]=
	\I[\sigma_3,Q_1(x,t)]Q_1(x,t).
	\end{equation}
\end{corollary}
\begin{proposition}\label{inv-t}
	Let $M(x,t,k)$ be the solution of the basic RH problem and assume that it has a partial derivative w.r.t. $t$ for some $x,t\in\mathbb{R}$.
	Then $M(x,t,k)\break e^{-2\I k^2t\sigma_3}$ satisfies the $t$-equation of the Lax pair \eqref{Lax} with $q(x,t)$ defined by \eqref{q(x,t)}.
\end{proposition}
\begin{proof}
	Introducing $\check{M}(x,t,k)=M(x,t,k)e^{-2\I k^2t\sigma_3}$ and arguing similarly as in\break Proposition \ref{inv-x} we arrive at
	\begin{equation}\label{M-t-eq}
	\begin{split}
	\partial_t\check M(x,t,k)\check M^{-1}(x,t,k)=&
	-2\I k^2\sigma_3+2\I [\sigma_3,Q_1(x,t)](kI-Q_1(x,t))\\
	&+2\I[\sigma_3,Q_2(x,t)],
	\end{split}
	\end{equation}
	for $k\in\mathbb{C}$, where $M(x,t,k)=I+\frac{Q_1(x,t)}{k}+\frac{Q_2(x,t)}{k^2}+\ord(k^{-3})$ as $k\to\infty$.
	Using \eqref{Q_2} we conclude that $\check M$ solves the $t$-equation in \eqref{Lax}.
\end{proof}

\section{Appendix B}
\label{AppB}
\begin{proposition}\cite{BC84}\label{r-S}
	Let $q_0(x)\in\mathcal{S}(\mathbb{R})$ and let the corresponding $a_j(k)$, $j=1,2$ satisfy Assumption 1.
	Then $r_j(k)\in\mathcal{S}(\mathbb{R})$, $j=1,2$.
\end{proposition}
\begin{proof}
	Follows from the representations of $a_j$ and $b$ obtained from the Neumann series of \eqref{psi-1-3} and \eqref{psi24}:
	\begin{equation}\label{a-b-int}
	a_j(k)=1+\sum\limits_{n=1}^{\infty}
	A_{j,n}(k),\quad j=1,2,\quad
	b(k)=\sum\limits_{n=0}^{\infty}
	B_n(k),
	\end{equation}
	where
	\begin{equation*}
	\begin{split}
	& A_{1,n}(k)=(-\sigma)^n\int_{-\infty}^{+\infty}d y_{1}
	\int_{-\infty}^{y_1}dy_2\dots
	\int_{-\infty}^{y_{2n-1}}dy_{2n}\\
	&\qquad\qquad\,\,\times\prod\limits_{j=1}^{n}
	q_0(y_{2j-1})\overline{q_0(-y_{2j})}\cdot
	\exp\{2\I k\sum\limits_{j=1}^{2n}(-1)^{j+1}y_j\},\\
	& A_{2,n}(k)=(-\sigma)^n\int_{-\infty}^{+\infty}d y_{1}
	\int_{-\infty}^{y_1}dy_2\dots
	\int_{-\infty}^{y_{2n-1}}dy_{2n}\\
	&\qquad\qquad\,\,\times\prod\limits_{j=1}^{n}
	\overline{q_0(-y_{2j-1})}q_0(y_{2j})\cdot
	\exp\{2\I k\sum\limits_{j=1}^{2n}(-1)^{j}y_j\},\\
	& B_0(k)=\int_{-\infty}^{+\infty}e^{-2\I ky}
	\sigma\overline{q_0(-y)}\,dy,\\
	& B_n(k)=(-\sigma)^n\int_{-\infty}^{+\infty}dy_1
	\int_{-\infty}^{y_1}dy_2\dots
	\int_{-\infty}^{y_{2n}}dy_{2n+1}
	\sigma \overline{q_0(-y_{2n+1})}\\
	&\qquad\qquad\,\,\times\prod\limits_{j=1}^{n}
	\overline{q_0(-y_{2j-1})}q_0(y_{2j})\cdot
	\exp\{2\I k\sum\limits_{j=1}^{2n+1}(-1)^{j}y_j\}.
	\end{split}
	\end{equation*}
\end{proof}

\section{Appendix C}
\label{AppC}
\begin{proposition}\label{Hss}\cite{DZ03, Z98}
	Consider the basic RH problem with $r_j(k)\in H^{1,1}(\mathbb{R})$ and without residue conditions.
	Suppose that this RH problem has an $L^2$ solution.
	Then $q(x,t)$ defined by \eqref{q(x,t)} also belongs to $H^{1,1}(\mathbb{R})$ for all fixed $t\in\mathbb{R}$.
\end{proposition}
\begin{proof}
	Step 1.
	Following the ideas of Theorem 2.1 in \cite{Z98}, we first prove that $q(\cdot,t)\in H^{0,1}(\mathbb{R})$.
	From \eqref{q(x,t)} and \eqref{M-int-eq} it follows that
	\begin{equation}\label{q-int}
	q(x,t)=-\frac{1}{\pi}\left(\int_{-\infty}^{+\infty}
	\mu(x,t,\zeta)(w^+(x,t,\zeta)+w^-(x,t,\zeta))\,d\zeta\right)_{12},
	\end{equation}
	where $w^{\pm}$ are defined in \eqref{upp-low}.
	Solvability of the RH problem is equivalent to existence of a bounded operator $(I-\mathcal{C}_w)^{-1}$ in $L^2$ (see Proposition 4.4 in \cite{Z89-1}), and \eqref{int-mu} implies that  $\mu(x,t,k)=I+(I-\mathcal{C}_w)^{-1}\mathcal{C}_w(I)$.
	Substituting this representation for $\mu$ in
	\eqref{q-int} one concludes that
	\begin{align}\label{q-int-1}
	\nonumber
	q(x,t)&=-\frac{\sigma}{\pi}\int_{-\infty}^{+\infty}r_2(\zeta)
	e^{-2\I\zeta x-4\I\zeta^2 t}\,d\zeta\\
	\nonumber
	&\quad\,-\frac{1}{\pi}\left(\int_{-\infty}^{+\infty}
	(I-\mathcal{C}_w)^{-1}\mathcal{C}_w(I)(w^+(x,t,\zeta)+w^-(x,t,\zeta))\,d\zeta\right)_{12}\\
	&=I_1(x,t)+I_2(x,t).
	\end{align}
	Since $r_j(k)\in H^{1,1}(\mathbb{R})$, the product $e^{\pm4\I k^2t}r_j(k)$ is also in $H^{1,1}(\mathbb{R})$ for all fixed $t$.
	Therefore,  $I_1(\cdot,t)\in H^{1,1}(\mathbb{R})$ by the Fourier transform.
	The second integral in \eqref{q-int-1} can be estimated by the Cauchy-Schwartz inequality:
	\begin{equation}\label{I-2-est}
	|I_2(x,t)|\leq\frac{1}{\pi}
	\|[(I-\mathcal{C}_w)^{-1}\mathcal{C}_w(I)]_{12}\|_{L^2}
	\|r_2(k)\|_{L^2}\leq
	C\|\left(\mathcal{C}_w(I)\right)_{12}\|_{L^2},\,\, C>0,
	\end{equation}
	where $(\mathcal{C}_w(I))_{12}=\mathcal{C}_+
	(\sigma r_2(\cdot)e^{-2\I (\cdot)x-4\I (\cdot)^2t})$.
	Using the inverse Fourier transform
	$
	f(k)=\frac{1}{\sqrt{2\pi}}\int_{-\infty}^{+\infty}
	e^{\I kz}\hat f(z)\,dz
	$
	and the following results of application of the Hilbert transform:
	$$
	H(e^{i\omega k})\equiv
	\frac{1}{\pi}\text{v.p.}\int_{-\infty}^{+\infty}
	\frac{e^{\I\omega\zeta}}{\zeta-k}\,d\zeta=
	\begin{cases}
	-\I e^{\I\omega k},&\omega>0,\\
	\I e^{\I\omega k},&\omega<0,
	\end{cases}
	$$
	we have for any $f\in L^2(\mathbb{R})$:
	\begin{align}\label{C_+f}
	\nonumber
	\mathcal{C}_+(f(k)e^{-2\I kx})&=
	\lim\limits_{\varepsilon\downarrow 0}\frac{1}{2\pi\I}
	\int_{-\infty}^{+\infty}
	\frac{f(\zeta)e^{-2\I\zeta x}}{\zeta-(k+\I\varepsilon)}\,d\zeta\\
	\nonumber
	&=
	\frac{1}{\sqrt{2\pi}}\int_{-\infty}^{+\infty}
	\hat f(z)
	\lim\limits_{\varepsilon\downarrow 0}\frac{1}{2\pi\I}
	\int_{-\infty}^{+\infty}
	\frac{e^{\I\zeta (z-2x)}}{\zeta-(k+\I\varepsilon)}
	\,d\zeta\,dz\\
	&=\frac{1}{\sqrt{2\pi}}
	\int_{-\infty}^{2x}\hat f(z)e^{\I k(z-2x)}\,dz.
	\end{align}
	Assuming that $f\in H^{1,0}(\mathbb{R})$,
	we have from \eqref{C_+f}, by the Plancherel theorem, that for $x\leq0$,
	\begin{equation}\label{C_+-norm}
	\begin{split}
	\|\mathcal{C}_+(f(k)e^{-2\I kx})\|^2_{L^2}=
	\int_{-\infty}^{2x}|\hat{f}(z)|^2\,dz&\leq
	\frac{1}{1+(2x)^2}
	\int_{-\infty}^{2x}(1+(2z)^2)|\hat{f}(z)|^2\,dz\\
	&\leq
	\frac{C\|f\|_{H^{1,0}}^2}{1+x^2},
	\end{split}
	\end{equation}
	which implies that $q(\cdot,t)\in H^{0,1}(-\infty,0)$.
	
	To prove that $q(\cdot,t)\in H^{0,1}(0,+\infty)$, we use another triangular factorization of $J$, cf.\,\,\eqref{upp-low} (we drop the argument in $r_j(k)$):
	\begin{equation}\label{low-upp}
	J(x,t,k)=
	\begin{pmatrix}
	1& 0\\
	\frac{r_1e^{2\I kx+4\I k^2t}}{1+\sigma r_1r_2}& 1\\
	\end{pmatrix}
	\begin{pmatrix}
	1+\sigma r_1r_2& 0\\
	0& \frac{1}{1+\sigma r_1r_2}\\
	\end{pmatrix}
	\begin{pmatrix}
	1& \frac{\sigma r_2e^{-2\I kx-4\I k^2t}}{1+\sigma r_1r_2}\\
	0& 1\\
	\end{pmatrix}.
	\end{equation}
	To get rid of the diagonal factor in \eqref{low-upp}, we introduce $\delta(k)$ by
	\begin{equation}
	\label{delta}
	\delta(k)=\exp\left\{\frac{1}{2\pi i}\int_{-\infty}^{+\infty}\frac{\ln(1+\sigma r_1(\zeta)r_2(\zeta))}{\zeta-k}\,d\zeta\right\},
	\end{equation}
	which  satisfies (i) the jump condition $\delta_+=\delta_-(1+\sigma r_1r_2)$ on $\mathbb{R}$ and (ii) the normalization condition $\delta\to1$ as $k\to\infty$.
	Notice that since the RH problem does not have zeros, $\int_{-\infty}^{+\infty}d\arg (1+\sigma r_1r_2)=0$ (see Proposition \ref{wn}) and the integral in \eqref{delta} converges.
	Then $\tilde M:=M\delta^{-\sigma_3}$ satisfies the jump condition $\tilde M_+=\tilde M_-\tilde J$ with
	\begin{equation}
	\begin{split}
	\tilde{J}(x,t,k)&=
	\begin{pmatrix}
	1& 0\\
	\frac{r_1(k)\delta_-^{-2}(k)e^{2\I kx+4\I k^2t}}
	{1+\sigma r_1(k)r_2(k)}& 1\\
	\end{pmatrix}
	\begin{pmatrix}
	1& \frac{\sigma r_2(k)\delta_+^{2}(k)e^{-2\I kx-4\I k^2t}}
	{1+\sigma r_1(k)r_2(k)}\\
	0& 1\\
	\end{pmatrix}\\
	&=(I-w^-(x,t,k))^{-1}(I+w^+(x,t,k)).
	\end{split}
	\end{equation}
	For such a factorization,
	$(\mathcal{C}_w(I))_{12}=\mathcal{C}_-(w^+_{12})$. Arguing similarly as in the case $x\leq 0$ and taking into account that (cf.\,\,\eqref{C_+f})
	\begin{equation}
	\mathcal{C}_-(f(k)e^{-2\I kx})=\frac{-1}{\sqrt{2\pi}}
	\int_{2x}^{+\infty}\hat f(z)e^{\I k(z-2x)}\,dz,
	\end{equation}
	we have (cf.\,\,\eqref{C_+-norm})
	\begin{equation}\label{C_--norm}
	\|\mathcal{C}_-(f(k)e^{-2\I kx})\|^2_{L^2}\leq
	\frac{C\|f\|_{H^{1,0}}^2}{1+x^2},\,x\geq 0,
	\end{equation}
	and thus $q(\cdot,t)\in H^{0,1}(0,+\infty)$.
	
	
	\noindent \textbf{step 2.} Now let us prove that $q(\cdot,t)\in H^{1,0}(\mathbb{R})$ (cf.\,\,Theorem 2.11 in \cite{Z98}).
	Since $\partial_x w^\pm(x,t,\cdot)\in L^2(\mathbb{R})\cap L^\infty(\mathbb{R})$, we have that $\partial_x(I-\mathcal{C}_w)\equiv-\mathcal{C}_{\partial_xw}$ is a bounded operator in $L^2$.
	Therefore, from the fact that
	$$
	\partial_x[(I-\mathcal{C}_w)^{-1}]=
	-(I-\mathcal{C}_w)^{-1}\cdot
	\partial_x(I-\mathcal{C}_w)\cdot
	(I-\mathcal{C}_w)^{-1},
	$$
	the existence of $\partial_x\mu(x,t,\cdot)\in L^2(\mathbb{R})$  follows.
	Using Proposition \ref{inv-x} and that $\mu=M_+(I+w^+)^{-1}$ we arrive at
	\begin{equation}
	\partial_x\mu(x,t,k)=
	-\I k[\sigma_3,\mu(x,t,k)]
	+\I[\sigma_3,U(x,t)]\mu(x,t,k),
	\end{equation}
	and thus (dropping below the arguments of the functions)
	\begin{equation}\label{diff-mu-1-x}
	\partial_x[\mu(w^++w^-)]=
	-\I k[\sigma_3,\mu(w^++w^-)]
	+\I[\sigma_3,U]\mu(w^++w^-),
	\end{equation}
	where  $w^\pm$ is defined by \eqref{upp-low}.
	The latter and \eqref{q-int} imply that $\partial_x q(x,t)\in L^2(\mathbb{R})$ w.r.t.\,\,$x$ for all $t$.
\end{proof}
\begin{corollary}\label{x-deriv-M}
	From \eqref{diff-mu-1-x} it follows that if $r_j\in H^{0,1}\cap L^\infty$, then the $L^2$ solution $M(x,t,k)$ of the RH problem has a partial derivative w.r.t.\,\,$x$.
\end{corollary}
Now let us prove that the inverse map $r_j\mapsto q$ is Lipschitz continuous.
\begin{proposition}\label{r-q-Lip}
	Under assumptions of Proposition \ref{Hss}, for any $t\in\mathbb{R}$ and $R>0$ there exists 
	$0<C_R(t)<C_R(1+T)$, 
	$C_R>0$, $t\in[-T,T]$ for any $T>0$,
	such that for all $r_j$, $\check{r}_j$ with $\max\{\|r_j\|_{H^{1,1}},\|\check r_j\|_{H^{1,1}}\}<R$, $j=1,2$,
	\begin{equation}\label{r-q-lip}
	\|q(\cdot,t)-\check q(\cdot,t)\|_{H^{1,1}(\mathbb{R})}\leq C_R(t)
	\|r_j(k)-\check{r}_j(k)\|_{H^{1,1}(\mathbb{R})}, \quad j=1,2.
	\end{equation}
	Here  $q$ and $\check{q}$ correspond to the reflection coefficients $r_j$ and $\check{r}_j$ respectively.
\end{proposition}
\begin{proof}
	From \eqref{q-int-1} we have that
	\begin{equation}\label{diff-q-est}
	\|q(\cdot,t)-\check q(\cdot,t)\|_{H^{1,1}}\leq
	\|I_1(\cdot,t)-\check I_1(\cdot,t)\|_{H^{1,1}}
	+\|I_2(\cdot,t)-\check I_2(\cdot,t)\|_{H^{1,1}},
	\end{equation}
	where $\check I_l$, $l=1,2$ are defined as   $I_l$, $l=1,2$ but with $r_j$ replaced by  $\check r_j$.
	Since the Fourier transform is continuous in $H^{1,1}$, it is enough to estimate the second term in r.h.s. of \eqref{diff-q-est} only.
	To this end, we observe that
	\begin{small}
		\begin{align}
		\nonumber
		\pi|I_2(x,t)-\check I_2(x,t)|&\leq\left|
		\int_{-\infty}^{+\infty}
		(I-\mathcal{C}_w)^{-1}\mathcal{C}_w(I)
		(w^+ + w^- -\check w^+ -\check w^-)(x,t,\zeta)\,d\zeta
		\right|\\
		\nonumber
		&\quad+\left|
		\int_{-\infty}^{+\infty}
		(I-\mathcal{C}_{w})^{-1}
		(\mathcal{C}_w(I)-\mathcal{C}_{\check{w}}(I))
		(\check w^+ +\check w^-)(x,t,\zeta)\,d\zeta
		\right|\\
		\nonumber
		&\quad+\left|
		\int_{-\infty}^{+\infty}
		[(I-\mathcal{C}_w)^{-1}-(I-\mathcal{C}_{\check w})^{-1}]
		\mathcal{C}_{\check w}(I)
		(\check w^+ +\check w^-)(x,t,\zeta)\,d\zeta
		\right|\\
		&\equiv|I_{21}(x,t)|+|I_{22}(x,t)|+|I_{23}(x,t)|.
		\end{align}
	\end{small}
	The first integral can be estimated as follows (cf. \eqref{I-2-est}):
	\begin{equation*}
	|I_{21}(x,t)|\leq\
	C\|\left(\mathcal{C}_w(I)\right)_{12}\|_{L^2}
	\|r_2-\check r_2\|_{L^2}\leq
	C\|\left(\mathcal{C}_w(I)\right)_{12}\|_{L^2}
	\|r_2-\check r_2\|_{H^{1,1}},\quad C>0.
	\end{equation*}
	Using \eqref{C_+-norm}, \eqref{C_--norm} and \eqref{diff-mu-1-x} we conclude that $\|I_{21}\|_{H^{1,1}}$ satisfies the required estimate.
	Taking into account that $\|\mathcal{C}_w(I)-\mathcal{C}_{\check{w}}(I)\|_{L^2}^2
	\leq C(1+T)^2\frac{\|r_2-\check r_2\|_{H^{1,1}}^2}{1+x^2}$ and using \eqref{diff-mu-1-x}, we have the needed estimate for $I_{22}$.
	Using the fact that
	$$
	\|\mathcal{C}_w-\mathcal{C}_{\check w}\|_{L^2}
	\leq C\|w-\check w\|_{L^{\infty}}\leq
	C\|w-\check w\|_{H^{1,1}},
	$$
	and that
	$$
	\|(I-A)^{-1}
	-(I-B)^{-1}\|_{L^2}
	\leq\|(I-A)^{-1}((I-B)-(I-A))(I-B)^{-1}\|_{L^2}
	\leq C\|A-B\|_{L^2},
	$$
	for any operators $A,B$ such that $(I-A)^{-1}$ and $(I-B)^{-1}$ exist and bounded, we have the Lipschitz continuity for $|I_{23}|$.
	
	Finally, to obtain \eqref{r-q-lip} for $j=1$, we use the inverse formula \eqref{q(-x,t)} involving the $(21)$ entry of $M$ (in contrast with the use of the  $(12)$ entry in \eqref{q(x,t)}).
\end{proof}

In \cite{BC84} Beals and Coifman proved that if $r_j\in \mathcal{S}$ and the RH problem is solvable, then the recovered potential $q(x,0)$ also belongs to the Schwartz space.
Here we consider the RH problem depending on $x,t$ and prove the following
\begin{proposition}\label{r-j-Schwartz}
	Consider the basic RH problem with $r_j(k)\in\mathcal{S}(\mathbb{R})$ and without residue conditions.
	Suppose that this RH problem has an $L^2$ solution.
	Then   $\partial_t^nq(\cdot,t)\in\mathcal{S}(\mathbb{R})$ exists for all $n\in\mathbb{N}\cup\{0\}$ and $t\in\mathbb{R}$.
\end{proposition}
\begin{proof}
	Notice that since $r_j\in\mathcal{S}$, the function $e^{\pm4\I k^2t}r_j$ belongs to $H^{s,s}(\mathbb{R})$ for any $s\in\mathbb{N}$.
	Arguing as in Step 1 of Proposition \ref{Hss} we obtain the following estimates (cf. \eqref{C_+-norm} and \eqref{C_--norm}):
	\begin{equation}
	\|\mathcal{C}_\pm(f(k)e^{-2\I kx})\|_{L^2}^2
	\leq\frac{C\|f\|^2_{H^{s,0}}}{(1+x^2)^s},\quad
	\pm x\leq 0,
	\end{equation}
	which implies that $q(\cdot,t)\in H^{0,s}(\mathbb{R})$ for any $s\in\mathbb{N}$.
	Then, taking into account that
	\begin{equation}\label{diff-mu}
	\begin{split}
	\partial_x^s[\mu(w^++w^-)]=&\,
	(-\I)^s k^s[\underbrace{
		\sigma_3,[\sigma_3,\dots,[\sigma_3,\mu(w^++w^-)]]}_s]\\
	&+\sum\limits_{l=0}^{s-1}k^{l}
	(c_{l,1}\mu(w^++w^-)+c_{l,2}\mu(w^++w^-)\sigma_3),
	\end{split}
	\end{equation}
	where $c_{l,j}$, $j=1,2$ depend on $Q,\dots\partial_x^{s-1} Q$, one obtains that $q(\cdot,t)\in H^{s,0}(\mathbb{R})$.
	Therefore, $q(\cdot,t)\in H^{s,s}(\mathbb{R})$ for any $s\in\mathbb{N}$, which implies that $q(\cdot,t)\in\mathcal{S}(\mathbb{R})$.
	
	From the $t$ equation in the Lax pair one obtains that
	\begin{equation}\label{diff-mu-t}
	\begin{split}
	\partial_t^n[\mu(w^++w^-)]=&\,
	(-2\I)^n k^{2n}[\underbrace{
		\sigma_3,[\sigma_3,\dots,[\sigma_3,\mu(w^++w^-)]]}_n]\\
	&+\sum\limits_{l=0}^{2n-1}k^{l}
	(d_{l,1}\mu(w^++w^-)+d_{l,2}\mu(w^++w^-)\sigma_3),
	\end{split}
	\end{equation}
	where $d_{l,j}$, $j=1,2$ depend on $\partial_x^m(Q,\dots,\partial_t^{n-1}Q)$, $m=0,1$.
	Therefore, using the smoothness and rapid decay of $w^\pm$ and reasoning  by induction on $n$, we conclude that
	$\partial_t^n q(\cdot,t)\in H^{s,s}(\mathbb{R})$ for any $s\in\mathbb{N}$, from which the statement of the proposition follows.
\end{proof}
\begin{corollary}\label{t-deriv-M}
	From \eqref{diff-mu-t} with $n=1$ it follows that if $r_j\in H^{0,2}\cap L^\infty$, then the $L^2$ solution $M(x,t,k)$ has a partial derivative w.r.t.\,\,$t$.
\end{corollary}
\begin{proposition}\label{M-fixed-k}
	The $L^2$ solution $M(x,t,k)$ of the basic RH problem without residue conditions satisfies the following properties:
	\begin{enumerate}
		\item if $r_j(k)\in H^{1,1}(\mathbb{R})$, then $M(\cdot,t,k_0)\in H^{1,1}(\mathbb{R})$ for any $k_0\in\mathbb{C}\setminus\mathbb{R}$ and $t\in\mathbb{R}$;
		\item if $r_j(k)\in \mathcal{S}(\mathbb{R})$, then $\partial_t^nM(\cdot,t,k_0)\in \mathcal{S}(\mathbb{R})$ for any $n\in\mathbb{N}\cup\{0\}$, $k_0\in\mathbb{C}\setminus\mathbb{R}$ and $t\in\mathbb{R}$.
	\end{enumerate}
	Moreover, for any fixed $k_0\in\mathbb{C}\setminus\mathbb{R}$ and $t\in\mathbb{R}$, $M(\cdot, t, k_0)$ is Lipschitz continuous, i.e.,
	for any $t\in\mathbb{R}$ and $R>0$ there exists $C_R(t)>0$ such that for all $r_j$, $\check{r}_j$ with $\max\{\|r_j\|_{H^{1,1}},\|\check r_j\|_{H^{1,1}}\}<R$, $j=1,2$ we have
	\begin{equation}\label{r-M-lip}
	\|M(\cdot,t,k_0)-\check M(\cdot,t,k_0)\|_{H^{1,1}(\mathbb{R})}\leq C_R(t)
	\|r_j(k)-\check{r}_j(k)\|_{H^{1,1}(\mathbb{R})}, \quad j=1,2.
	\end{equation}
	Here $M$ and $\check M$ are the solutions of RH problem wih reflection coefficients $r_j$ and $\check r_j$ respectively.
\end{proposition}
\begin{proof}
	Using the representation
	(see \eqref{M-int-eq})
	\begin{equation}
	M(x,t,k_0)-I=\frac{1}{2\pi\I}\int_{-\infty}^{+\infty}
	\frac{\mu(x,t,\zeta)(w^+(x,t,\zeta)+w^-(x,t,\zeta))}
	{\zeta-k_0}\,d\zeta,
	\end{equation}
	and taking into account that
	$\frac{w^\pm(x,t,\cdot)}{(\cdot)-k_0} \in H^{1,1}(\mathbb{R})$ for $r_j(k)\in H^{1,1}(\mathbb{R})$, the arguments in the proof of Proposition \ref{Hss} imply claim 1.
	Then observing that $\frac{w^\pm(x,t,\cdot)}{(\cdot)-k_0} \in \mathcal{S}(\mathbb{R})$ for $r_j(k)\in \mathcal{S}(\mathbb{R})$ and arguing similarly as in the proof of
	Proposition \ref{r-j-Schwartz}, we arrive at claim 2.
	Moreover,
	arguing  as in the proof of Proposition \ref{r-q-Lip}, we obtain \eqref{r-M-lip}.
\end{proof}
\begin{remark}\label{cont-M}
	Observe that for $r_j(k)\in H^{1,1}(\mathbb{R})$, the matrix $M(x,t,k_0)$ is continuous w.r.t.\,\,$x,t$ for all fixed $k_0\in\mathbb{C}\setminus\{\mathbb{R}\}$.
	Indeed, taking $r_{j,n}\in\mathcal{S}$ such that $r_{j,n}\to r_j$ in $H^{1,1}$ and using
	\eqref{r-M-lip} we have
	$$
	\sup\limits_{t\in[-T,T]}
	\|M(\cdot,t,k_0)-M_n(\cdot,t,k_0)\|_{H^{1,1}(\mathbb{R})}\to 0,\quad n\to\infty,
	$$
	where $M_n$ is the solution of the RH problem involving the reflection coefficients $r_{j,n}$ (notice that since small perturbations of an invertible operator is invertible,  the RH problem remains solvable for small $H^{1,1}$ perturbations of $r_j$).
	Then using Morrey's inequality
	we conclude that $M_n(\cdot,\cdot, k_0)\to M(\cdot,\cdot, k_0)$ in $C(\mathbb{R}\times[-T,T])$ and $M(\cdot,\cdot, k_0)$ is continuous as a limit of continuous functions.
\end{remark}
\end{appendices}

\end{document}